\documentclass[11pt]{article}

\usepackage[T1]{fontenc}
\usepackage[utf8]{inputenc}
\usepackage{lmodern}
\usepackage{amsmath,amssymb,amsthm}
\usepackage{geometry}
\usepackage{enumitem}
\usepackage{microtype}
\usepackage[hidelinks]{hyperref}
\usepackage{tikz}

\setlist{itemsep=2pt,topsep=3pt,parsep=0pt,partopsep=0pt}

\newtheorem{theorem}{Theorem}[section]
\newtheorem{proposition}[theorem]{Proposition}
\newtheorem{lemma}[theorem]{Lemma}
\newtheorem{example}[theorem]{Example}
\newtheorem{corollary}[theorem]{Corollary}
\theoremstyle{definition}
\newtheorem{definition}[theorem]{Definition}
\newtheorem{remark}[theorem]{Remark}
\newtheorem*{theoremA}{Theorem A (local reduced case)}
\newtheorem*{theoremB}{Theorem B (Extension theorem)}
\newtheorem*{theoremC}{Theorem C (local general case dimension two)}
\newtheorem*{theoremD}{Theorem D (Homogeneous logarithmic divisor theorem)}
\newtheorem*{theoremE}{Theorem E}
\newtheorem*{theoremF}{Theorem F}
\newtheorem*{theoremG}{Theorem G}
\newtheorem*{theoremH}{Theorem H (local two terms dimension m)}

\newcommand{\C}{\mathbb C}
\newcommand{\N}{\mathbb N}
\newcommand{\R}{\mathbb R}
\newcommand{\Om}{\mathcal O}
\newcommand{\Mm}{\mathcal M}

\title{Relative Cohomology and Deformations\\
of Logarithmic Foliations on Complex Spaces}
\author{V\'ictor Le\'on \and Bruno Sc\'ardua}
\author{
V\'ictor Le\'on%
\thanks{ILACVN--CICN, Universidade Federal da Integração Latino-Americana,
Parque Tecnológico Itaipu, Foz do Iguaçu--PR, 85867-970, Brazil.
E-mail: \href{mailto:victor.leon@unila.edu.br}{victor.leon@unila.edu.br}}
\and
Bruno Sc\'ardua%
\thanks{Instituto de Matem\'atica, Universidade Federal do Rio de Janeiro,
CP 68530, Rio de Janeiro--RJ, 21945-970, Brazil.
E-mail: \href{mailto:bruno.scardua@gmail.com}{bruno.scardua@gmail.com}}
}\date{}

\begin{document}

\setlength{\abovedisplayskip}{7pt plus 2pt minus 3pt}
\setlength{\belowdisplayskip}{7pt plus 2pt minus 3pt}
\setlength{\abovedisplayshortskip}{4pt plus 2pt minus 2pt}
\setlength{\belowdisplayshortskip}{5pt plus 2pt minus 2pt}

\maketitle

\begin{abstract}
We study relative cohomology for logarithmic differential forms and
meromorphic forms that are relatively closed with respect to the associated
logarithmic foliation. Under suitable Diophantine, geometric, and topological
hypotheses, we obtain decompositions into a meromorphic multiple of the
defining logarithmic form, an exact meromorphic form, and a logarithmic form
with constant residues. We establish a meromorphic extension theorem from
suitable two-dimensional sections and prove local, polynomial, and
homogeneous versions of the relative-cohomology decomposition; the global
polynomial result is proved in arbitrary dimension by ambient leafwise
continuation and meromorphic extension. Resolution
of singularities, non-nodal saturation, and holonomy gluing are used to
treat singular logarithmic foliations. As an application, we derive formal
normal forms for analytic integrable deformations, including a
several-component result in dimension two and a two-component extension in
higher dimensions.
\end{abstract}



% Outline appearing on handwritten pages 2--3:
% 1. Introduction
% 2. Preliminaries: non-singular case; diophantine conditions
% 3. The relative cohomology equation for logarithmic simple singularities
% 4. The extension theorem
% 5. The case of simple polyhomogeneous forms
% 6. Topology and dynamics of the leaves:
%    6.1 General position and diophantine conditions
%    6.2 Nodal singularities
% 7. Relative cohomology for homogeneous forms of type
%    omega_0=dQ/Q-lambda dP/P
% 8. Relative cohomology for germs of logarithmic 1-forms
% 9. Analytic integrable deformations of logarithmic foliation germs
\noindent\textbf{Funding Declaration:} During the preparation of this work, the second listed author was partially funded by 
Fundação Getúlio Vargas - Rio de Janeiro.  
\section{Introduction}

In this work we study relative cohomology for complex logarithmic
$1$-forms. More precisely, let $f_1,\ldots,f_n$ be holomorphic functions on a
complex manifold $M$ and $\lambda_1,\ldots,\lambda_n\in\C$ complex numbers.
We consider
\[
 \omega_0=d\log\bigl(f_1^{\lambda_1}\cdots f_n^{\lambda_n}\bigr)
 =\sum_{j=1}^n\lambda_j\frac{df_j}{f_j},
\]
the corresponding logarithmic $1$-form on $M$ with residues
$\lambda_1,\ldots,\lambda_n$.

We denote by $\mathcal F_{\omega_0}$ the codimension-one holomorphic foliation
with singularities on $M$ defined by
\[
 \Omega_0=\left(\prod_{i=1}^n f_i\right)\omega_0
 =\sum_{i=1}^n\lambda_i f_1\cdots\widehat{f_i}\cdots f_n\,df_i.
\]

Let $\eta$ be a meromorphic $1$-form on $M$. We shall say that $\eta$ is
\emph{closed with respect to $\omega_0$} if
\[
 d\eta\wedge\omega_0=0
\]
off the poles of $\eta$ and $\omega_0$. We say that $\eta$ is
\emph{exact with respect to $\omega_0$} if there is a meromorphic function
$h$ on $M$ such that
\[
 \eta\wedge\omega_0=dh\wedge\omega_0.
\]
This last expression is related to the expression  $\eta = a \omega_0 + dh$ as a consequence of 
Saito-de Rham division lemma (\cite{deRham1954},\cite{saito}). 
Our aim is to study the answer to the following basic question:

\begin{quote}
Given $\eta$ closed with respect to $\omega_0$, is $\eta$ exact with respect
to $\omega_0$?
\end{quote}

More generally:

\begin{quote}
If $d\eta\wedge\omega_0=0$, what can be said about $\eta$?
\end{quote}

This question is strongly related to the study of deformations
$\{\omega_t\}_{t\in(\C,0)}$ of $\omega_0$ by integrable $1$-forms $\omega_t$.
It is also connected to the construction of formal and convergent normal
forms.

The question of relative cohomology has already been addressed by several
authors. Its origins go back to the work of Ilyashenko~\cite{IlyashenkoYakovenko}. In
Berthier--Cerveau~\cite{BerthierCerveau} the authors consider the case of holomorphic $1$-forms
$\eta$ and prove various triviality results related to integral conditions.
This is also the framework found in Berthier--Loray~\cite{BerthierLoray}, where the $1$-form
$\omega_0$ is a germ of resonant non-linearizable $1$-form at the origin
$0\in\C^2$. Again the authors find some integral conditions on paths on the
leaves of $\omega_0$.

In this work, where $\omega_0$ is linear logarithmic with generic conditions
of non-resonance and diophantine type on the residues
$\lambda_1,\ldots,\lambda_n$, the open (nonclosed) leaves of
$\mathcal F_{\omega_0}$ contain no nontrivial cycles. Thus we do not work with
integral conditions. On the other hand, some genericity conditions on the
residues allow the solution of the relative cohomology equation
\[
 \widetilde\eta\wedge\omega_0=dh\wedge\omega_0
\]
for suitable modifications $\widetilde\eta=\eta+\cdots$ of $\eta$, on small
neighborhoods of singular points. Then the simple connectivity of the open
leaves will, always under some mild genericity conditions on the divisor
\[
 D(\omega_0):=\left(\prod_{j=1}^n f_j=0\right)\subset M,
\]
allow the extension and gluing of the local solutions. The modification
$\widetilde\eta$ is obtained by removing from $\eta$ its logarithmic part.
Our first result in the local framework is the following.

\begin{theoremA}
{\em Let $\lambda\in\C\setminus\R$ satisfy the positive-cone diophantine condition
\[
 |\lambda m+n|>\frac{c}{(m+n)^\delta},\qquad
 \forall(m,n)\in\N^2\setminus\{(0,0)\},
\]
and let
\[
 \omega_0=\frac{dy}{y}-\lambda\frac{dx}{x}.
\]
Let $\eta$ be a meromorphic $1$-form germ at $0\in\C^2$ such that $xy\eta$ is
holomorphic. There are $\alpha,\beta\in\C$ and a meromorphic germ
$h\in\Mm_2$ such that:
\begin{enumerate}[label=\textup{(\arabic*)}]
 \item $xyh\in\Om_2$;
 \item
 \[
  \left(\eta-\alpha\frac{dx}{x}-\beta\frac{dy}{y}\right)\wedge\omega_0
  =dh\wedge\omega_0.
 \]
\end{enumerate}
In particular, there is $a\in\Mm_2$, with $xya\in\Om_2$, such that
\[
 \eta=a\omega_0+\alpha\frac{dx}{x}+\beta\frac{dy}{y}+dh.
\]
}
\end{theoremA}

For
$\lambda=(\lambda_1,\ldots,\lambda_n)\in \mathbb C^n$ and
$k=(k_1,\ldots,k_n)\in\mathbb Z^n$, put
$|k|=|k_1|+\cdots+|k_n|$. We say that the residues satisfy the
\emph{integer-index Diophantine
condition} if there are $c>0$ and $\delta\geq0$ such that
\[
 \left|\sum_{j=1}^n k_j\lambda_j\right|
 \geq\frac{c}{(1+|k|)^\delta},
 \qquad k\in\mathbb Z^n\setminus\{0\}.
\]

\begin{definition}[general position germs]
\label{def:two-component-general-position}
Let $F,G\in\mathcal O_m$ be reduced germs vanishing at the origin. We say
that $F$ and $G$ are in \emph{general position at the origin} if
\[
 dF(0)\wedge dG(0)\neq0.
\]
Equivalently, $(F,G):(\C^m,0)\to(\C^2,0)$ is a submersion. In particular,
there are local coordinates $(x,y,z_3,\ldots,z_m)$ in which $F=x$ and
$G=y$.
\end{definition}

Let
\(
 \Omega=\sum_{j=1}^m A_j\,dz_j
\)
be a germ of holomorphic $1$-form at $0\in\C^m$. We say that $\Omega$ is
\emph{reduced} if its coefficients have no nonunit common factor.
Equivalently, its zero set has codimension at least two. A divisor germ
$D$ at $0\in\C^m$ is called \emph{reduced} if it admits a local equation
$D=(f=0)$ in which $f$ has no repeated irreducible factor. If $\Omega$ is
reduced and integrable and defines a foliation $\mathcal F$, we say that
$D$ is \emph{$\mathcal F$-invariant} if
\[
 f\ \text{divides}\ \Omega\wedge df.
\]
Equivalently, the regular part of every irreducible component of $D$ is
tangent to $\mathcal F$.

\begin{definition}[general position embedding]
\label{def:section-general-position}
Let $\Omega$ be a reduced holomorphic integrable $1$-form defining a
codimension-one foliation $\mathcal F$ at $0\in\C^m$, and let $D$ be a
reduced $\mathcal F$-invariant divisor. An embedding
\[
 \zeta:(\C^2,0)\longrightarrow(\C^m,0)
\]
is said to be in \emph{general position with respect to $(\mathcal F,D)$}
if:
\begin{enumerate}[label=\textup{(\roman*)}]
 \item $\zeta^{-1}(\operatorname{Sing}\mathcal F)=\{0\}$ and
 $\operatorname{Sing}(\zeta^*\Omega)$ has codimension two in $(\C^2,0)$;
 \item for every $p\neq0$ at which $\mathcal F$ is regular,
 \[
  d\zeta_p(T_p\C^2)\not\subset T_{\zeta(p)}\mathcal F,
 \]
 or equivalently $\zeta^*\Omega(p)\neq0$;
 \item $\zeta$ is transverse, outside the origin, to every smooth stratum
 of $D$.
\end{enumerate}
This definition is independent of the chosen reduced integrable form
defining $\mathcal F$.
\end{definition}

If $F$ and $G$ are in general position at the origin in the sense of
Definition~\ref{def:two-component-general-position}, the submersion theorem
gives local coordinates $(x,y,z_3,\ldots,z_m)$ in which $F=x$ and $G=y$.
In these coordinates, the embedding
\[
 \zeta(x,y)=(x,y,0,\ldots,0)
\]
is in general position with respect to the associated two-component
logarithmic foliation and its invariant divisor, in the sense of
Definition~\ref{def:section-general-position}. Moreover,
\[
 F\circ\zeta=x,\qquad G\circ\zeta=y.
\]

We also prove a meromorphic extension theorem allowing the passage from a
two-dimensional section to the local dimension $m\geq3$ case. Its proof is
an adaptation of the extension argument in~\cite{BerthierCerveau}.

\begin{theoremB}
\phantomsection\label{thm:B}
{\em Let
\(
 \omega=\sum_{j=1}^n\lambda_j\frac{df_j}{f_j}
\)
be a germ of logarithmic $1$-form at $0\in\C^m$, $m\geq2$, where
$D=(\prod_jf_j=0)$ is reduced and $\lambda_j\in\C^*$ for every $j$.
Let $\eta$ be a germ of meromorphic $1$-form at
$0\in\C^m$ such
that:
\begin{enumerate}[label=\textup{(\arabic*)}]
 \item $d\eta\wedge\omega=0$;
 \item the polar set of $\eta$ is contained in $D$, and $\eta$ has finite
 pole order along every irreducible component of $D$.
\end{enumerate}
Let
\[
 \zeta_0:(\C^2,0)\longrightarrow(\C^m,0)
\]
be an embedding in general position with respect to
$(\mathcal F_\omega,D)$ in the sense of
Definition~\ref{def:section-general-position}. Suppose that
there exist $a_0,h_0\in\Mm_2$ and constants $c_1,\ldots,c_n\in\C$ such that
the polar sets of $a_0,h_0$ are contained in the restricted divisor, both
functions have finite pole order there, and
\[
 \zeta_0^*\eta
 =a_0\zeta_0^*\omega+dh_0
 +\sum_{j=1}^n c_j
 \frac{d(f_j\circ\zeta_0)}{f_j\circ\zeta_0}.
\]
Then there exist $a,h\in\Mm_m$ such that:
\begin{enumerate}[label=\textup{(\arabic*)}]
 \item $a\circ\zeta_0=a_0$ and $h\circ\zeta_0=h_0$;
 \item the polar sets of $a,h$ are contained in $D$, and both functions
 have finite pole order along $D$;
 \item
 \[
  \eta=a\omega+dh+\sum_{j=1}^n c_j\frac{df_j}{f_j}.
 \]
\end{enumerate}
If, moreover, with $P=\prod_{j=1}^n f_j$,
\[
 P\eta\in\Omega_m^1,
 \qquad (P\circ\zeta_0)a_0\in\mathcal O_2,
 \qquad (P\circ\zeta_0)h_0\in\mathcal O_2,
\]
then $a,h$ may be chosen so that
\[
 Pa\in\mathcal O_m,\qquad Ph\in\mathcal O_m.
\]
}
\end{theoremB}

Using Theorem~A, induction along the resolution of singularities of
$\mathcal F_{\omega_0}$ at $(\C^2,0)$, and gluing techniques inspired by
Berthier--Cerveau~\cite{BerthierCerveau} and adapted to the present
meromorphic framework, we obtain:

\begin{theoremC}
{\em Let
\[
 \omega_0=\sum_{j=1}^n\lambda_j\frac{df_j}{f_j}
\]
be a germ of non-dicritical logarithmic form at $0\in\C^2$, whose
residues $\lambda_1,\ldots,\lambda_n$ are $\mathbb Z$-linearly independent
and satisfy the integer-index diophantine condition.
Assume that $D=(\prod_{j=1}^nf_j=0)$ is reduced.
Suppose that the open leaves of $\mathcal F_{\omega_0}$ are simply connected.
Assume also that the reduction of $\mathcal F_{\omega_0}$ is free of nodal
separators.
Let $\eta$ be a germ of meromorphic $1$-form at $0\in\C^2$ such that
\[
 \left(\prod_{j=1}^n f_j\right)\eta
\]
is holomorphic and $d\eta\wedge\omega_0=0$. Then there are meromorphic
germs $a,h\in\Mm_2$ and constants $c_1,\ldots,c_n\in\C$ such that
\[
 \eta=a\omega_0+dh+\sum_{j=1}^n c_j\frac{df_j}{f_j}.
\]
}
\end{theoremC}

For $n=2$, the simple connectivity hypothesis is assured when $f_1,f_2$
are in general position at $0\in\C^2$ in the sense of
Definition~\ref{def:two-component-general-position} (cf. the proof of  Lemma~\ref{lemma:simplyconnectedreduced}).

We next record an auxiliary homogeneous rigidity result. It is used later
to treat homogeneous deformations with a fixed invariant divisor and also
serves as the model for the several-component homogeneous statement.

If $R\in\C[z_1,\ldots,z_m]$ is reduced and $\Omega$ is a polynomial
$1$-form, we say that the hypersurface $(R=0)$ is
\emph{invariant by $\Omega$} if
\[
 R\ \text{divides}\ \Omega\wedge dR.
\]
This definition does not require $\Omega$ to be integrable. When $\Omega$
is integrable and defines a foliation, it is equivalent to the usual
invariance of the regular part of $(R=0)$ by that foliation.

\begin{definition}
\label{def:general-projective-position}
Let $P,Q\in\C[z_1,\ldots,z_m]$ be reduced irreducible homogeneous
polynomials. We say that $P$ and $Q$ are in \emph{general projective
position} if the projective hypersurfaces
\[
 \mathbb P(P=0),\qquad \mathbb P(Q=0)\subset\mathbb{CP}^{m-1}
\]
are smooth and intersect transversely.
\end{definition}

\begin{theoremD}
{\em Let $m\geq3$ and let $P,Q\in\C[z_1,\ldots,z_m]$ be reduced irreducible
homogeneous polynomials of the same degree and in general projective
position. If $\Omega$ is a homogeneous
polynomial $1$-form such that
\(
 \deg\Omega=2\deg P-1
\)
and both $(P=0)$ and $(Q=0)$ are invariant by $\Omega$, then
\[
 \frac{\Omega}{PQ}=a\frac{dP}{P}+b\frac{dQ}{Q}
\]
for some constants $a,b\in\C$.}
\end{theoremD}

We now return to the relative-cohomology problem in the global polynomial
setting. The local decomposition can be propagated directly along the
ambient leaves and then globalized by the Levi extension theorem. This
gives the following result in arbitrary dimension.

\begin{theoremE}
\phantomsection\label{thm:E}
{\em Let $m\geq2$ and let
$P,Q\in\C[z_1,\ldots,z_m]$ be irreducible polynomials. Assume that the
union of their projective closures with the hyperplane at infinity
$H_\infty$ is a simple normal-crossings divisor in $\mathbb{CP}^m$ and
that the projective closures of $(P=0)$ and $(Q=0)$ meet in the affine
part. Let
$\lambda\in\C\setminus\mathbb R$ satisfy the integer-index
diophantine condition
\[
 |\lambda i+j|>\frac{c}{(|i|+|j|)^\delta},\qquad
 \forall(i,j)\in\mathbb Z^2\setminus\{(0,0)\},
\]
and let
\[
 \omega_0=\frac{dQ}{Q}-\lambda\frac{dP}{P}.
\]
Assume that the open affine leaves of $\omega_0$ in $\C^m$ are simply
connected.

Given a polynomial $1$-form $\Omega$ on $\C^m$, let
$\eta=\Omega/(PQ)$. Then the following conditions are equivalent:
\begin{enumerate}[label=\textup{(\arabic*)}]
 \item $d\eta\wedge\omega_0=0$;
 \item there are rational functions
 $a,h\in\C(z_1,\ldots,z_m)$ and constants
 $\alpha,\beta\in\C$ such that
 \[
  \eta=a\omega_0+dh+\alpha\frac{dP}{P}+\beta\frac{dQ}{Q}.
 \]
\end{enumerate}
}
\end{theoremE}

We shall use the following geometric hypothesis in the several-component
homogeneous result.

\begin{definition}[very general position germs]
\label{def:very-general-position}
Let $P_1,\ldots,P_n\in\C[z_1,\ldots,z_m]$ be homogeneous polynomials. We
shall say that $P_1,\ldots,P_n$ are in \emph{very general position} if:
\begin{enumerate}[label=\textup{(\roman*)}]
\item each polynomial $P_j$ has singularities in an analytic set of
codimension at least $3$;
\item the polynomials $P_j$ are of ordinary-crossings type off an analytic
subset of codimension at least $3$: more precisely, there is an analytic
set $Z\subset\C^m$ of codimension at least $3$ such that, for every
$p\notin Z$, the differentials
\[
 \{dP_j(p):P_j(p)=0\}
\]
are linearly independent.
\end{enumerate}

Given germs $f_1,\ldots,f_n\in\Om_m$, we shall also say that they are in {\it very general
position} if their first nonzero homogeneous jets
$P_1,\ldots,P_n$, respectively, are in very general position in the
preceding sense.
\end{definition}

If $\eta=\Omega/(F_1\cdots F_n)$, we say that
$D=(F_1\cdots F_n=0)$ is \emph{$\eta$-invariant} when every irreducible
component $(F_j=0)$ is invariant by $\Omega$ in the sense introduced
before Theorem~D.

We finally obtain the following result for homogeneous $1$-forms:

\begin{theoremF}
\phantomsection\label{thm:F}
{\em 
Let $F_1,\ldots,F_n\in\C[z_1,\ldots,z_m]$ be reduced irreducible homogeneous
polynomials of the same degree $\nu$.
Assume either that $m=2$ and $n=2$, or that $m\geq3$ and
$F_1,\ldots,F_n$ are in very general position. Let $\Omega$ be a
homogeneous polynomial $1$-form of degree
$\deg\Omega=n\nu-1$ and let
\[
 \eta=\frac{\Omega}{F_1\cdots F_n}.
\]
Assume that $D=(\prod_{j=1}^nF_j=0)$ is $\eta$-invariant.
Then there are $c_1,\ldots,c_n\in\C$ such that
\[
 \eta=\sum_{j=1}^n c_j\frac{dF_j}{F_j}.
\]
}
\end{theoremF}

In the final part we consider analytic deformations of logarithmic foliations
from the viewpoint of relative cohomology. We start with the local case,
obtaining:

We write $\mathcal M_m[[t]]$ for the formal power-series ring over the field
of meromorphic germs at $0\in\C^m$. Membership in this ring does not impose
a uniform bound on the pole orders of the coefficients; these orders may
increase with the power of $t$.

\begin{theoremG}
\phantomsection\label{thm:G}
{\em Let
\[
 \omega_0=\sum_{j=1}^n\lambda_j\frac{df_j}{f_j}
\]
be a non-dicritical logarithmic foliation germ at $0\in\C^2$. Assume that
$D=(f_1\cdots f_n=0)$ is reduced and that the residues are
$\mathbb Z$-linearly independent. Assume also that there exist constants
$c>0$ and $\delta\geq0$ such that
\[
 \left|\sum_{j=1}^n k_j\lambda_j\right|
 \geq\frac{c}{(1+|k|)^\delta},
 \qquad k=(k_1,\ldots,k_n)\in\mathbb Z^n\setminus\{0\},
\]
where $|k|=|k_1|+\cdots+|k_n|$. Suppose that the open leaves are simply
connected and that the reduction of $\mathcal F_{\omega_0}$ is free of
nodal separators. Put
\[
 P=f_1\cdots f_n,\qquad \Omega_0=P\omega_0.
\]
Given an analytic deformation by integrable holomorphic $1$-forms
\[
 \Omega_t=\Omega_0+\sum_{k=1}^{\infty}t^k\Omega_k,
\]
there are a formal unit 
\(
 \widehat U=\prod_{k=1}^{\infty}(1+t^ka_k)\in\mathcal M_2[[t]],
\)
a formal meromorphic function
\[
 \widehat H=\sum_{k=1}^{\infty}t^kh_k\in\mathcal M_2[[t]],
\]
and formal series
\[
 \widehat c_j(t)=\sum_{k=1}^{\infty}c_j^kt^k\in t\C[[t]],
 \qquad j=1,\ldots,n,
\]
such that
\[
 \frac{\Omega_t}{P\widehat U}
 =\omega_0+d\widehat H+
 \sum_{j=1}^n\widehat c_j(t)\frac{df_j}{f_j}.
\]
Equivalently,
\[
 \frac{\Omega_t}{P\widehat U}
 =d\widehat H+
 \sum_{j=1}^n\bigl(\lambda_j+\widehat c_j(t)\bigr)
 \frac{df_j}{f_j}.
\]
Every coefficient $a_k,h_k$ has finite pole order supported on $D$.
}
\end{theoremG}

The next theorem is a higher-dimensional two-component result obtained
through the finite-pole relative decomposition and the Extension Theorem.

\begin{theoremH}
\phantomsection\label{thm:H}
{\em Let $F,G\in\mathcal O_m$, $m\geq3$, be reduced irreducible germs in general position at the origin.
Let $\lambda\in\C\setminus\mathbb Q$ and assume that there exist constants
$c>0$ and $\delta\geq0$ such that
\[
 |\lambda n+m|\geq\frac{c}{(|n|+|m|)^\delta},
 \qquad (n,m)\in\mathbb Z^2\setminus\{(0,0)\}.
\]
Put
\[
 \omega_0=\frac{dG}{G}-\lambda\frac{dF}{F}.
\]
Given an analytic deformation of
\[
 \Omega_0=F\,dG-\lambda G\,dF
\]
by integrable $1$-forms
\[
 \Omega_t=\Omega_0+\sum_{k=1}^{+\infty}t^k\Omega_k,
\]
there are:
\begin{itemize}
 \item a formal unit
 \[
 \widehat U:=\prod_{k=1}^{+\infty}(1+t^k a_k)
 \in\mathcal M_m[[t]],
  \qquad a_k\in\Mm_m;
 \]
 \item a formal function
 \[
 \widehat H:=\sum_{k=1}^{+\infty}t^k h_k
 \in\mathcal M_m[[t]],
  \qquad h_k\in\Mm_m;
 \]
 \item formal functions $\widehat\alpha(t),\widehat\beta(t)
 \in\widehat\Om_1$ satisfying
 \[
  \widehat\alpha(0)=-\lambda,\qquad
  \widehat\beta(0)=1.
 \]
\end{itemize}
such that
\[
 \Omega_t=\widehat U\left[
 \widehat\alpha(t)G\,dF+\widehat\beta(t)F\,dG
 +FG\,d\widehat H
 \right],
\]
where
\[
 d\widehat H:=\sum_{k=1}^{+\infty}t^k dh_k.
\]
}
\end{theoremH}

\section{The relative cohomology equation for logarithmic simple singularities}

We consider in complex dimension two:
\begin{itemize}
 \item a linear logarithmic $1$-form
 \(
  \omega_0=\frac{dy}{y}-\lambda\frac{dx}{x},\qquad \lambda\neq0;
 \)
 \item a meromorphic $1$-form $\eta(x,y)$ such that $xy\eta(x,y)$ is
 holomorphic;
 \item the equation
 \(
  \eta\wedge\omega_0=dh\wedge\omega_0,
 \)
 where $h$ is a meromorphic function $h=h(x,y)$.
\end{itemize}
We will discuss under which conditions the equation
\[
 \eta\wedge\omega_0=dh\wedge\omega_0
\]
has a solution $h(x,y)$. 

\subsection{Examples}
We motivate our results with some examples.
\begin{example}{\rm 
We investigate the equation $\eta\wedge\omega_0=dh\wedge\omega_0$ for
\[
 \omega_0=\frac{dy}{y}-\lambda\frac{dx}{x},\qquad
 \eta=\frac{dy}{y}-\mu\frac{dx}{x}.
\]
We have
\[
 \eta\wedge\omega_0
 =\left(\frac{dy}{y}-\mu\frac{dx}{x}\right)
  \wedge\left(\frac{dy}{y}-\lambda\frac{dx}{x}\right)
 =(\lambda-\mu)\frac{dx\wedge dy}{xy}.
\]
Write $h=\sum h_{nm}x^ny^m$. Then
\begin{align*}
 dh\wedge\omega_0
 &=\sum h_{nm}\left(nx^{n-1}y^m\,dx+mx^ny^{m-1}\,dy\right)
 \wedge\left(\frac{dy}{y}-\lambda\frac{dx}{x}\right)\\
 &=\sum h_{nm}x^ny^m(n+m\lambda)\frac{dx\wedge dy}{xy}.
\end{align*}
Therefore $\eta\wedge\omega_0=dh\wedge\omega_0$ is equivalent to
\[
 \frac{\lambda-\mu}{xy}
 =\sum h_{nm}x^ny^m\frac{n+m\lambda}{xy},
\]
hence $\lambda-\mu=h_{00}(0+0\lambda)$, which is equivalent to
$\lambda=\mu$.}
\end{example}

\begin{example}
{\rm 

Assume $\lambda\neq1$ and let $\eta=a\,dy/x+b\,dx/y$. Then
\[
 \eta\wedge\omega_0
 =\left(\frac{\lambda a}{x^2}+\frac{b}{y^2}\right)dx\wedge dy.
\]
Thus $\eta\wedge\omega_0=dh\wedge\omega_0$ if and only if
\[
 \frac{\lambda a}{x^2}+\frac{b}{y^2}
 =\sum(n+m\lambda)h_{nm}x^{n-1}y^{m-1},
\]
which is equivalent to
\[
 \lambda a=(-1+\lambda)h_{-1,1},\qquad
 b=(1-\lambda)h_{1,-1},
\]
and $h_{nm}=0$ for every $(n,m)\notin\{(-1,1),(1,-1)\}$. Equivalently,
\[
 h=\frac{\lambda a}{-1+\lambda}x^{-1}y
 +(1-\lambda)^{-1}bxy^{-1}.
\]
}
\end{example}

\begin{example}
{\rm 
Let $\eta=y\,dx/x$. Then
\[
 \eta\wedge\omega_0=\frac{dx\wedge dy}{x},
 \qquad
 dh\wedge\omega_0
 =\sum h_{nm}\frac{x^ny^m}{xy}(n+m\lambda)\,dx\wedge dy.
\]
Therefore $\eta\wedge\omega_0=dh\wedge\omega_0$ if and only if
\[
 \sum(n+m\lambda)h_{nm}x^ny^{m-1}=1.
\]
Thus a solution is
\[
 h=\frac{y}{\lambda}.
\]

Similarly, $\eta=x\,dy/y$ admits the solution
\[
 h=\lambda x.
\]
}
\end{example}

\begin{example}
{\rm 
Let $\eta=dx/(xy)$. Then
\[
 \eta\wedge\omega_0=\frac{dx\wedge dy}{xy^2}.
\]
Thus $\eta\wedge\omega_0=dh\wedge\omega_0$ if and only if
\[
 \sum h_{nm}x^ny^{m+1}(n+m\lambda)=1.
\]
A solution is therefore
\[
 h=-\frac{1}{\lambda y}.
\]
Similarly, $\eta=dy/(xy)$ admits the solution
\[
 h=-\frac{\lambda}{x}.
\]
}
\end{example}

\begin{example}
{\rm 
Assume $\lambda\notin\mathbb Q$ and let
$\eta=x^{m_0}\,dx/y$, $m_0\geq1$. Then
\[
 \eta\wedge\omega_0=\frac{x^{m_0}}{y^2}\,dx\wedge dy.
\]
The equation $\eta\wedge\omega_0=dh\wedge\omega_0$ is equivalent to
\[
 \sum h_{nm}(n+m\lambda)x^ny^{m+1}=x^{m_0+1}.
\]
A solution is
\[
 h=\frac{x^{m_0+1}}{(m_0+1-\lambda)y}.
\]
Similarly, $\eta=y^{m_0}\,dy/x$, $m_0\geq1$, admits the solution
\[
 h=\frac{\lambda y^{m_0+1}}
 {(\lambda(m_0+1)-1)x}.
\]

So, for $1$-forms $\eta$ like
\begin{align*}
 \eta={}&a_{-1,0}\frac{dx}{x}+a_{-1,-1}\frac{dx}{xy}
 +a_{0,-1}\frac{dx}{y}+a_{m,-1}\frac{x^m\,dx}{y}\\
 &+b_{-1,0}\frac{dy}{x}+b_{-1,-1}\frac{dy}{xy}
 +b_{0,-1}\frac{dy}{y}+b_{-1,m}\frac{y^m\,dy}{x},
\end{align*}
the obstruction to the equation $\eta\wedge\omega=dh\wedge\omega$ having a
meromorphic solution is given by the logarithmic part
\[
 \eta_0=a_{-1,0}\frac{dx}{x}+b_{0,-1}\frac{dy}{y},
\]
which must satisfy $\eta_0\wedge\omega=0$, i.e.,
\[
 a_{-1,0}+\lambda b_{0,-1}=0.
\]
In other words, $\eta_0=t\omega$ for some $t\in\C$.
}
\end{example}
\begin{example}
{\rm 
Let
\[
 \eta=\frac{x^2\,dy-y\,dx}{xy}=x\frac{dy}{y}-\frac{dx}{x}.
\]
Then
\[
 \eta\wedge\omega_0=(\lambda x-1)\frac{dx\wedge dy}{xy}.
\]
Hence $\eta\wedge\omega_0=dh\wedge\omega_0$ would require
\[
 \lambda x-1=\sum h_{nm}x^ny^m(n+m\lambda),
\]
and therefore
\[
 h_{nm}=0,\quad \forall(n,m)\notin\{(0,0),(1,0)\},
 \qquad h_{1,0}=\lambda,
 \qquad 0h_{0,0}=-1,
\]
which is not possible. The obstruction comes from the fact that
$\eta=x\,dy/y-dx/x$ has a logarithmic part $dx/x$ which is not a multiple of
$\omega=dy/y-\lambda\,dx/x$.
}
\end{example}

\subsection{Diophantine condition}

We recall a classical definition for holomorphic diffeomorphisms and vector
fields, especially in connection with linearization problems. The one for vector fields
goes as follows. Given complex numbers $\lambda_1,\ldots,\lambda_n\in\C$, we
say that they satisfy the \emph{positive-index diophantine condition} if
there are constants $C>0$, $\delta>0$ such that
\begin{align*}
 \left|m_1\lambda_1+\cdots+m_n\lambda_n-\lambda_j\right|
 &\geq \frac{C}{(m_1+\cdots+m_n)^\delta},\\
 &\qquad \forall j\in\{1,\ldots,n\},\quad
 \forall(m_1,\ldots,m_n)\in\N^n,\\
 &\hspace{45mm}m_1+\cdots+m_n\geq2.
\end{align*}
For the pair $(\lambda,1)$ and nonnegative indices we shall use the
\emph{positive-cone diophantine condition}
\[
 |i\lambda+j|\geq\frac{C}{(i+j)^\delta},
 \qquad \forall(i,j)\in\N^2\setminus\{(0,0)\}.
\]
This condition is automatically verified if $\lambda\in\C\setminus\R$. We
use this version when only nonnegative indices occur.

When Laurent expansions or homogeneous rational functions introduce negative
indices, we use the stronger \emph{integer-index diophantine condition}
\[
 |i_1\lambda_1+\cdots+i_n\lambda_n|
 >\frac{C}{(|i_1|+\cdots+|i_n|)^\delta},
 \qquad \forall(i_1,\ldots,i_n)\in\mathbb Z^n\setminus\{0\},
\]
and
\[
 |\lambda i+j|>\frac{C}{(|i|+|j|)^\delta},
 \qquad \forall(i,j)\in\mathbb Z^2\setminus\{(0,0)\},
\]
Irrationality excludes resonances, but does not
by itself imply this integer-index estimate.

We recall that  $\lambda_1,\ldots,\lambda_n$ are {\it $\mathbb Z$-linearly
independent} if 
\[
 i_1\lambda_1+\cdots+i_n\lambda_n=0
 \quad\text{for some }(i_1,\ldots,i_n)\in\mathbb Z^n
\]
if and only if $i_1=\cdots=i_n=0$.

\subsection{Solving the relative cohomology equation}
Let us now solve the relative cohomology equation. 
We write
\[
 \eta=\sum_{n,m\geq-1}a_{nm}x^ny^m\,dx
 +\sum_{n,m\geq-1}b_{nm}x^ny^m\,dy,
 \qquad
 \omega_0=\frac{dy}{y}-\lambda\frac{dx}{x},
 \qquad
 h=\sum h_{nm}x^ny^m.
\]
We have
\[
 \eta\wedge\omega_0
 =\sum_{n,m\geq-1}a_{nm}x^ny^{m-1}\,dx\wedge dy
 +\lambda\sum_{n,m\geq-1}b_{nm}x^{n-1}y^m\,dx\wedge dy.
\]
Separating the boundary terms, this becomes
\begin{align*}
 \eta\wedge\omega_0
 =\Bigg[{}&a_{-1,-1}x^{-1}y^{-2}
 +\sum_{n\geq0}a_{n,-1}x^ny^{-2}
 +a_{-1,0}x^{-1}y^{-1}\\
 &+\sum_{n\geq0}a_{n,0}x^ny^{-1}
 +\sum_{\substack{n\geq0\\m\geq1}}a_{nm}x^ny^{m-1}\\
 &+\lambda\Bigg(
 b_{-1,-1}x^{-2}y^{-1}
 +\sum_{m\geq0}b_{-1,m}x^{-2}y^m
 +\sum_{m\geq0}b_{0,m}x^{-1}y^m\\
 &\hspace{30mm}+b_{0,-1}x^{-1}y^{-1}
 +\sum_{\substack{n\geq1\\m\geq0}}b_{nm}x^{n-1}y^m
 \Bigg)\Bigg]dx\wedge dy.
\end{align*}
Rewrite
\[
 \sum_{\substack{n\geq0\\m\geq1}}a_{nm}x^ny^{m-1}
 +\lambda\sum_{\substack{n\geq1\\m\geq0}}b_{nm}x^{n-1}y^m
 =\sum_{n,m\geq0}(a_{n,m+1}+\lambda b_{n+1,m})x^ny^m.
\]
Then
\begin{align*}
 \eta\wedge\omega_0=\Bigg[{}&
 a_{-1,-1}x^{-1}y^{-2}
 +\lambda b_{-1,-1}x^{-2}y^{-1}
 +(a_{-1,0}+\lambda b_{0,-1})x^{-1}y^{-1}\\
 &+\left(\sum_{n\geq0}a_{n,-1}x^n\right)y^{-2}
 +\lambda\left(\sum_{m\geq0}b_{-1,m}y^m\right)x^{-2}\\
 &+\left(\sum_{n\geq0}a_{n,0}x^n\right)y^{-1}
 +\lambda\left(\sum_{m\geq0}b_{0,m}y^m\right)x^{-1}\\
 &+\sum_{n,m\geq0}(a_{n,m+1}+\lambda b_{n+1,m})x^ny^m
 \Bigg]dx\wedge dy.
\end{align*}

On the other hand,
\begin{align*}
 dh\wedge\omega_0
 &=\left(\sum_{n,m}nh_{nm}x^{n-1}y^m\,dx
 +\sum_{n,m}mh_{nm}x^ny^{m-1}\,dy\right)
 \wedge\left(\frac{dy}{y}-\lambda\frac{dx}{x}\right)\\
 &=\sum_{n,m}(n+m\lambda)h_{nm}x^{n-1}y^{m-1}\,dx\wedge dy.
\end{align*}
Separating the boundary terms gives
\begin{align*}
 dh\wedge\omega_0=\Bigg[{}&
 (-1-\lambda)h_{-1,-1}x^{-2}y^{-2}
 +\sum_{m\geq0}(-1+m\lambda)h_{-1,m}x^{-2}y^{m-1}\\
 &+\sum_{n\geq0}(n-\lambda)h_{n,-1}x^{n-1}y^{-2}
 +\sum_{m\geq1}\lambda m h_{0,m}x^{-1}y^{m-1}\\
 &+\sum_{n\geq1}n h_{n,0}x^{n-1}y^{-1}
 +\sum_{\substack{n\geq1\\m\geq1}}(n+\lambda m)h_{nm}x^{n-1}y^{m-1}
 \Bigg]dx\wedge dy.
\end{align*}

Thus $\eta\wedge\omega_0=dh\wedge\omega_0$ is equivalent to the following
relations:
\begin{align*}
 (-1-\lambda)h_{-1,-1}x^{-2}y^{-2}&=0,\\
 a_{-1,-1}x^{-1}y^{-2}&=-\lambda h_{0,-1}x^{-1}y^{-2},\\
 \lambda b_{-1,-1}x^{-2}y^{-1}&=-h_{-1,0}x^{-2}y^{-1},\\
 \sum_{n\geq0}a_{n,-1}x^ny^{-2}
 &=\sum_{n\geq1}(n-\lambda)h_{n,-1}x^{n-1}y^{-2},\\
 \lambda\sum_{m\geq0}b_{-1,m}x^{-2}y^m
 &=\sum_{m\geq1}(-1+m\lambda)h_{-1,m}x^{-2}y^{m-1},\\
 a_{-1,0}x^{-1}y^{-1}&=0\cdot h_{0,0}x^{-1}y^{-1},\\
 \sum_{n\geq0}a_{n,0}x^ny^{-1}
 &=\sum_{n\geq1}nh_{n,0}x^{n-1}y^{-1},\\
 b_{0,-1}x^{-1}y^{-1}&=\lambda\,0\,h_{0,0}x^{-1}y^{-1},\\
 \lambda\sum_{m\geq0}b_{0,m}x^{-1}y^m
 &=\sum_{m\geq1}\lambda m h_{0,m}x^{-1}y^{m-1},\\
 \sum_{n,m\geq0}(a_{n,m+1}+\lambda b_{n+1,m})x^ny^m
 &=\sum_{n,m\geq1}(n+\lambda m)h_{n,m}x^{n-1}y^{m-1}.
\end{align*}
Therefore,
\begin{align*}
 (-1-\lambda)h_{-1,-1}&=0,\\
 a_{-1,-1}&=-\lambda h_{0,-1},\\
 \lambda b_{-1,-1}&=-h_{-1,0},\qquad a_{-1,0}=0,\\
 a_{m,-1}&=(m+1-\lambda)h_{m+1,-1},\qquad \forall m\geq0,
 \qquad b_{0,-1}=0,\\
 \lambda b_{-1,m}&=(-1+(m+1)\lambda)h_{-1,m+1},\qquad \forall m\geq0,\\
 a_{m,0}&=(m+1)h_{m+1,0},\qquad \forall m\geq0,\\
 \lambda b_{0,m}&=\lambda(m+1)h_{0,m+1},\qquad \forall m\geq0,\\
 a_{n,m+1}+\lambda b_{n+1,m}
 &=(n+1+\lambda(m+1))h_{n+1,m+1},\qquad \forall n,m\geq0.
\end{align*}
From now on we assume that $\lambda\notin\mathbb Q$. The solution is given
by, provided that $a_{-1,0}=b_{0,-1}=0$,
\begin{align*}
 h_{-1,-1}&=0,\\
 h_{0,-1}&=-\frac{a_{-1,-1}}{\lambda},\\
 h_{-1,0}&=-\lambda b_{-1,-1},\\
 h_{m+1,-1}&=\frac{a_{m,-1}}{m+1-\lambda},\qquad \forall m\geq0,\\
 h_{-1,m+1}&=\frac{\lambda b_{-1,m}}{-1+(m+1)\lambda},\qquad \forall m\geq0,\\
 h_{m+1,0}&=\frac{a_{m,0}}{m+1},\qquad \forall m\geq0,\\
 h_{0,m+1}&=\frac{\lambda b_{0,m}}{\lambda(m+1)},\qquad \forall m\geq0,\\
 h_{n+1,m+1}&=
 \frac{a_{n,m+1}+\lambda b_{n+1,m}}{n+1+\lambda(m+1)},\qquad \forall n,m\geq0.
\end{align*}

\begin{proof}[Proof of Theorem A]
The coefficient formulas above define a formal meromorphic solution $h$
with at most simple poles along $(xy=0)$.
Now we prove the convergence of the power series
\[
 h=\sum_{n,m=-1}^{+\infty}h_{nm}x^ny^m
\]
as defined above. For this we observe that:
\begin{enumerate}[label=\textup{(\arabic*)}]
 \item Since $\lambda\notin\mathbb Q$, there is $\varepsilon>0$ such that
 $|\lambda-m|\geq\varepsilon$ for all $m\in\mathbb N$. Therefore
 \[
  \left|\frac{a_{m,-1}}{m+1-\lambda}\right|
  \leq |a_{m,-1}|\frac1\varepsilon,\qquad \forall m\geq0.
 \]
 \item Since
 \[
  \frac1{|\frac{-1}{\lambda}+m+1|}\longrightarrow0
  \quad\text{as }m\longrightarrow\infty,
 \]
 we conclude that, for some $C>0$,
 \[
  \left|\frac{\lambda b_{-1,m}}{-1+(m+1)\lambda}\right|
  \leq C|b_{-1,m}|,\qquad \forall m\geq0.
 \]
 \item
 \[
  \left|\frac{a_{m,0}}{m+1}\right|\leq|a_{m,0}|,\qquad \forall m\geq0.
 \]
 \item
 \[
  \left|\frac{b_{0,m}}{m+1}\right|\leq|b_{0,m}|,\qquad \forall m\geq0.
 \]
 \item Finally,
 \begin{align*}
 \left|
 \frac{a_{n,m+1}+\lambda b_{n+1,m}}
 {n+1+\lambda(m+1)}
 \right|
 &\leq
 \frac{|a_{n,m+1}|}{|n+1+\lambda(m+1)|}\\
 &\quad+
 \frac{|\lambda|\,|b_{n+1,m}|}{|n+1+\lambda(m+1)|}.
 \end{align*}
 From the diophantine condition we have
 \[
  |\lambda(m+1)+n+1|\geq
  \frac{C}{(m+n+2)^\delta},
  \qquad \forall n,m\in\mathbb N,
 \]
 and therefore
 \begin{align*}
 \left|
 \frac{a_{n,m+1}+\lambda b_{n+1,m}}
 {(n+1)+\lambda(m+1)}
 \right|
 &\leq |a_{n,m+1}|\frac{(n+m+2)^\delta}{C}\\
 &\quad+|\lambda|\,|b_{n+1,m}|\frac{(n+m+2)^\delta}{C}.
 \end{align*}
\end{enumerate}
Using \textup{(1)}--\textup{(5)}, the Cauchy--Hadamard criterion, and the
convergence of the series
\[
 \sum a_{nm}x^ny^m,\qquad \sum b_{nm}x^ny^m,
\]
we conclude that $h$ is convergent.

For completeness, put
\[
 \theta=\eta-a_{-1,0}\frac{dx}{x}-b_{0,-1}\frac{dy}{y}.
\]
The identity $(\theta-dh)\wedge\omega_0=0$ gives a meromorphic function
$a$ such that $\theta-dh=a\omega_0$. If
$\theta-dh=A\,dx+B\,dy$, then
\[
 a=yB=-\frac{xA}{\lambda}.
\]
The coefficient construction above gives $h_{-1,-1}=0$ and
$h=\sum_{i,j\geq-1}h_{ij}x^iy^j$. Hence $y\,\partial h/\partial y$ has no
exponent smaller than $-1$ in either variable. Since $xy\theta$ is
holomorphic, the equality $a=yB$ shows that $xya$ is holomorphic. This
proves the additional pole-order assertion in Theorem~A.
\end{proof}

\section{The Extension Theorem}

In this section we show how to pass from dimension two to dimension $m\geq3$
in our study of relative cohomology, using the general-position terminology
of Definition~\ref{def:section-general-position}.

\begin{lemma}[Flow-box extension]
Let $F:(U,0)\to(\C,0)$ be a holomorphic submersion and let
$\omega=g\,dF$, where $g\in\mathcal O(U)^\times$. Let $S\subset U$ be a
two-dimensional submanifold such that $F|_S$ is a submersion. Suppose that
$\theta$ is either holomorphic, or meromorphic with a pole of finite order
supported on $(F=0)$, and satisfies
\[
 d\theta\wedge dF=0
\]
and that
\[
 \theta|_S=a_0\omega|_S+dh_0
\]
where $a_0,h_0$ are holomorphic in the first case and meromorphic with
finite pole order along $(F|_S=0)$ in the second. Then, after shrinking
$U$, there are functions $a,h$ of the same respective type such that
\[
 \theta=a\omega+dh,\qquad a|_S=a_0,\qquad h|_S=h_0.
\]
\end{lemma}

\begin{proof}
Choose coordinates $(u,z_2,\ldots,z_m)$ with $u=F$ and
$S=(z_3=\cdots=z_m=0)$. Write
\[
 \theta=A\,du+\sum_{j=2}^m B_j\,dz_j.
\]
The condition $d\theta\wedge du=0$ says that
$\sum_{j=2}^mB_jdz_j$ is closed on every plaque $(u=\text{constant})$.
For $z=(z_2,\ldots,z_m)$, define
\[
 h(u,z):=h_0(u,z_2)+
 \int_{(z_2,0,\ldots,0)}^{(z_2,z_3,\ldots,z_m)}
 \sum_{j=2}^mB_j(u,\zeta)\,d\zeta_j.
\]
On a sufficiently small polydisc the integral is path-independent. It is
holomorphic in the holomorphic case. In the meromorphic case, Laurent
expansion in $u$ shows that integration in the $z$ variables preserves the
finite lower bound on the powers of $u$; hence $h$ is meromorphic with
finite pole order along $(u=0)$. Then $\theta-dh$ vanishes on every plaque,
hence equals $a\omega$, with $a$ of the same type. Restriction to $S$ gives
the required normalization.
\end{proof}

\begin{proof}[Proof of Theorem B]
Put
\[
 P=\prod_{j=1}^n f_j,\qquad
 \Omega=P\omega,\qquad
 \theta=\eta-\sum_{j=1}^n c_j\frac{df_j}{f_j}.
\]
Thus $\Omega$ is a holomorphic integrable $1$-form defining
$\mathcal F_\omega$, the polar divisor of $\theta$ is supported on $D$ and
has finite order, and
\[
 d\theta\wedge\Omega=0,\qquad
 \zeta_0^*\theta
 =\frac{a_0}{P\circ\zeta_0}\,\zeta_0^*\Omega+dh_0.
\]

After a holomorphic change of coordinates, suppose that
$\zeta_0(x)=(x,0)$ in $\C^2\times\C^{m-2}$. If $m=2$, the embedding is a
local biholomorphism, and transporting $a_0,h_0$ by its inverse gives the
conclusion directly. Assume henceforth that $m\geq3$ and choose concentric
bidiscs
\[
 \Delta_r^2\Subset\Delta_{R_1}^2\Subset\Delta_{R_2}^2
\]
and let
\[
 K=\overline{\Delta_{R_1}^2\setminus\Delta_r^2}\times\{0\}.
\]
General position implies that $K$ is disjoint from
$\operatorname{Sing}\mathcal F_\omega$. It also implies that every point
of $K\cap D$ is a smooth point of exactly one component of $D$: a singular
point of a component, or an intersection of two components, is a zero of
$\Omega$. Since these exceptional sets are closed and $K$ is compact,
after decreasing the radii there is a transverse polydisc
$\Delta_\rho^{m-2}$ on which the foliation remains regular and $D$ has at
most one smooth component near every point of the product crown.

Fix $p\in K$. If $p\notin D$, apply the holomorphic flow-box extension
lemma to $\Omega$. If $p\in D$, use a flow box in which its unique local
component is $(u=0)$. Since that component is invariant, $u$ may be chosen
as a local first integral, and the meromorphic flow-box extension lemma
applies. In either case we obtain, on a product neighborhood $W_p$ of $p$,
meromorphic functions $A_p,H_p$, of finite pole order supported on $D$,
such that
\[
 \theta=A_p\Omega+dH_p,\qquad
 A_p|_S=\frac{a_0}{P|_S},\qquad H_p|_S=h_0.
\]

These normalized extensions are unique. Indeed, on an overlap write
$\Omega=g\,dF$. If two such pairs are defined there, then
\[
 d(H_p-H_q)=-(A_p-A_q)\Omega,
\]
so $H_p-H_q=\varphi(F)$ for a one-variable meromorphic function $\varphi$.
The difference vanishes on $S$, and $F|_S$ is a submersion; hence
$\varphi=0$, first off $D$ and then meromorphically across it. Thus
$H_p=H_q$ and $A_p=A_q$. Choosing the $W_p$ in product form ensures that
every component of a nonempty overlap meets $S$.

A finite subcover of $K$ therefore gives meromorphic functions $A',H'$ on
a product crown
\[
 \left(\Delta_{R_1}^2\setminus\overline{\Delta_r^2}\right)
 \times\Delta_{\rho'}^{m-2}
\]
satisfying $\theta=A'\Omega+dH'$. Compactness of a slightly smaller closed
crown gives a uniform bound for their pole orders. Hence, for some $N$,
\[
 P^N A',\qquad P^N H'
\]
are holomorphic on that product crown.

The parameterized Hartogs theorem, applied in the first two variables,
extends these two holomorphic functions to
$\Delta_{R_1}^2\times\Delta_{\rho'}^{m-2}$. Denote the extensions by
$\widetilde A,\widetilde H$ and set
\[
 A=\frac{\widetilde A}{P^N},\qquad
 H=\frac{\widetilde H}{P^N}.
\]
Then $A,H$ are meromorphic with finite pole order supported on $D$. The
meromorphic identity principle extends $\theta=A\Omega+dH$ from the crown
to the full polydisc. Since $\Omega=P\omega$, setting $a=PA$ and $h=H$
gives
\[
 \eta=a\omega+dh+\sum_{j=1}^n c_j\frac{df_j}{f_j}.
\]
On the punctured section these functions restrict to $a_0,h_0$; the
meromorphic identity principle on $S$ gives
$a\circ\zeta_0=a_0$ and $h\circ\zeta_0=h_0$ at the origin as well.

Under the additional simple-pole assumptions in the statement, the local
flow-box construction gives $PH_p$ and $P^2A_p$ holomorphic. Indeed,
integration in the plaque variables preserves the simple pole of $H_p$,
whereas differentiating $H_p$ may increase the pole order by one in the
coefficient $A_p$ relative to $\Omega=P\omega$. Consequently, the glued
functions $PH'$ and $P^2A'$ are holomorphic on the product crown. Applying
the parameterized Hartogs theorem to these two functions and dividing the
extensions by $P$ and $P^2$, respectively, gives a decomposition with
\[
 Ph\in\mathcal O_m,\qquad Pa=P^2A\in\mathcal O_m.
\]
The same restriction and identity-principle argument preserves the
normalizations $a\circ\zeta_0=a_0$ and $h\circ\zeta_0=h_0$.
\end{proof}

Using Theorems~A and~B we obtain:

\begin{proposition}
Let $F,G\in\Om_m$ be irreducible germs in general position at $0\in\C^m$
in the sense of Definition~\ref{def:two-component-general-position},
$m\geq2$. Let $\lambda\in\C\setminus\mathbb Q$ satisfy the integer-index
diophantine condition
\[
 |\lambda n+m|>\frac{c}{(|n|+|m|)^\delta},
 \qquad \forall(n,m)\in\mathbb Z^2\setminus\{(0,0)\}.
\]
If we put
\[
 \omega_0=\frac{dG}{G}-\lambda\frac{dF}{F}
\]
and $\eta$ is a meromorphic $1$-form germ
\[
 \eta=\frac{\eta_0}{FG}
\]
for some holomorphic $1$-form germ $\eta_0$. Assume, in addition, that
\[
 d\eta\wedge\omega_0=0.
\]
Then there is a germ of
meromorphic function $h\in\Mm_m$, with polar divisor of finite order
supported on $(FG=0)$, such that
\[
 \left(\eta-a\frac{dF}{F}-b\frac{dG}{G}\right)\wedge\omega_0
 =dh\wedge\omega_0
\]
for some constants $a,b\in\C$.
\end{proposition}

\section{Relative Cohomology: The homogeneous case}
In this section we prove our main results for homogeneous 1-forms. 
\begin{proof}[Proof of Theorem D]
Put $\eta=\Omega/(PQ)$. At a smooth point of $(P=0)$ outside $(Q=0)$,
invariance gives
\[
 \Omega=A\,dP+P\beta,
\]
so $\eta=(A/Q)dP/P+\beta/Q$. The analogous expression holds along
$(Q=0)$, and transversality gives the same conclusion at their
intersection. Thus $\eta$ is logarithmic along the reduced
normal-crossings divisor $(PQ=0)$.

The residue of $\eta$ along $(P=0)$ is a holomorphic homogeneous function
of degree zero on the punctured cone. It descends to a holomorphic function
on the smooth irreducible projective hypersurface $(P=0)\subset
\mathbb{CP}^{m-1}$ and is therefore a constant $a$. Similarly the residue
along $(Q=0)$ is a constant $b$. Hence
\[
 \theta:=\eta-a\frac{dP}{P}-b\frac{dQ}{Q}
\]
has zero residues and extends holomorphically across $(PQ=0)$ on
$\C^m\setminus\{0\}$. Since $m\geq3$, Hartogs extension gives a holomorphic
$1$-form at the origin. Its coefficients are homogeneous of degree $-1$,
so they vanish. Therefore $\theta=0$.
\end{proof}

\begin{proof}[Proof of Theorem F]
Put $D=(F_1\cdots F_n=0)$. Since $D$ is reduced and invariant by $\eta$,
each irreducible component $(F_j=0)$ is invariant. Outside the analytic
subset where the components are singular or fail to have normal crossings,
the local divisibility argument used in the proof of Theorem~D shows
that $\eta$ is logarithmic along $D$.

Its residue $r_j$ along $(F_j=0)$ is homogeneous of degree zero. It descends
to a holomorphic function on the smooth locus of the irreducible projective
hypersurface
\[
 (F_j=0)\subset\mathbb{CP}^{m-1}.
\]
When $m\geq3$, the omitted subset has codimension at least two in this
hypersurface, so $r_j$ extends and is constant; write $r_j=c_j$. When
$m=2$ and $n=2$, irreducibility makes $F_1,F_2$ linear and the same
conclusion is immediate. Therefore
\[
 \theta=\eta-\sum_{j=1}^n c_j\frac{dF_j}{F_j}
\]
has zero residues and extends holomorphically across $D$ away from the
codimension-three exceptional subset. Hartogs extension removes that subset
and the origin. The coefficients of $\theta$ are homogeneous of degree
$-1$, hence vanish. Thus
\[
 \eta=\sum_{j=1}^n c_j\frac{dF_j}{F_j},
\]
which proves \hyperref[thm:F]{Theorem~F}.
\end{proof}

\section{Topology and dynamics of the leaves: general position, nodal separators  and dicriticalness}

\subsection{Simple-connectivity of open leaves}
We use the notion of very general position introduced in
Definition~\ref{def:very-general-position}. We now present its consequences
for the topology of the open leaves of the associated logarithmic
foliation.

Let
\[
 \omega_0=\sum_{j=1}^n\lambda_j\frac{df_j}{f_j}
\]
be a logarithmic foliation on a complex manifold $M$. By an \emph{open
leaf} of $\mathcal F_{\omega_0}$ we shall mean a leaf $L_0\subset M$ of
the foliation $\mathcal F_{\omega_0}$ defined by $\omega_0$ on
\[
 M\setminus\bigcup_{j=1}^n(f_j=0),
\]
i.e., a leaf not contained in the hypersurfaces $(f_j=0)$.

\begin{proposition}
\label{prop:local-simple-connectivity}
Let $m\geq3$ and
\[
 \omega_0=\sum_{j=1}^n\lambda_j\frac{dP_j}{P_j},
\]
where:
\begin{enumerate}[label=\textup{(\roman*)}]
\item $P_1,\ldots,P_n$ are homogeneous polynomials having singularities in
an analytic set of codimension at least $3$;
\item $P_1,\ldots,P_n$ are of ordinary-crossings type off an analytic set
of codimension at least $3$;
\item $\omega_0$ is non-dicritical;
\item $\lambda_1,\ldots,\lambda_n$ satisfy a diophantine condition 
\[
 |j_1\lambda_1+\cdots+j_n\lambda_n|
 >\frac{c}{(|j_1|+\cdots+|j_n|)^\delta}
\]
for every $(j_1,\ldots,j_n)\in\mathbb Z^n$ such that
$j_1\lambda_1+\cdots+j_n\lambda_n\neq0$;
\item $(\lambda_1,\ldots,\lambda_n,1)$ are $\mathbb Z$-linearly independent.
\end{enumerate}
Then the open leaves of $\mathcal F_{\omega_0}$ in
$U\setminus\operatorname{Sing}(\mathcal F_{\omega_0})$ are simply connected in $U$, for
$U$ a sufficiently small neighborhood of $0$ in $\C^m$.
\end{proposition}

\begin{proof}
This is Theorem~4.1.1 of Berthier--Cerveau~\cite{BerthierCerveau}.
Their proof reduces to dimension three by a Lefschetz argument, uses the
abelianity of the fundamental group of the complement of the projectivized
normal-crossings divisor, and computes the projective holonomy. The
$\mathbb Z$-linear independence of
$(\lambda_1,\ldots,\lambda_n,1)$ makes the lifted leafwise loops close only
for the trivial homotopy class.
\end{proof}

Using this, we can prove:

\begin{proposition}
Let $m\geq3$ and
\[
 \omega_0=\frac{dQ}{Q}-\lambda\frac{dP}{P},
\]
where:
\begin{enumerate}[label=\textup{(\roman*)}]
\item $P,Q\in\C[z_1,\ldots,z_m]$ are homogeneous polynomials having
singularities in an analytic set of codimension at least $3$;
\item $P,Q$ are of ordinary-crossings type off an analytic set of
codimension at least $3$;
\item $\omega_0$ is non-dicritical;
\item $\lambda,1$ satisfy a diophantine condition and are
$\mathbb Z$-linearly independent.
\end{enumerate}
Then the open leaves of the affine foliation $\mathcal F_{\omega_0}$ are
simply connected in $\C^m$.
\end{proposition}

\begin{proof}
For $t\in\C^*$ let $\delta_t(z)=tz$. Homogeneity gives
$\delta_t^*\omega_0=\omega_0$. Hence $\delta_t$ is an automorphism of the
foliation and
\[
 \delta_t(L_z)=L_{tz}.
\]
From Proposition~\ref{prop:local-simple-connectivity} there is a
neighborhood $U$ of the origin of $\C^m$
such that every non-separatrix leaf of $\mathcal F_{\omega_0}$ in
$U\setminus\operatorname{Sing}(\mathcal F_{\omega_0})$ is simply connected in $U$.
Let $\gamma$ be a loop in a global open leaf $L$. Since its image is
compact, $\delta_t\circ\gamma$ is contained in $U$ for $|t|>0$
sufficiently small. It lies in one local leaf $L_t$ and is null-homotopic
there. If $H:\mathbb D\to L_t$ is a contraction, then
$\delta_{t^{-1}}\circ H$ is a contraction of $\gamma$ in $L$. Thus every
global open leaf is simply connected.
\end{proof}

\begin{remark}
Let $P,Q\in\mathcal O_2$ be distinct irreducible germs vanishing at the
origin and let $\lambda\in\C\setminus\mathbb Q$. Then
\[
 \omega_0=\frac{dQ}{Q}-\lambda\frac{dP}{P}
\]
defines a non-dicritical foliation germ at $0\in\C^2$.

Indeed, consider any sequence of point blow-ups resolving the reduced
divisor $(PQ=0)$. If $C$ is an exceptional component, then
\[
 \operatorname{Res}_C(\pi^*\omega_0)
 =m_C(Q)-\lambda m_C(P),
\]
where
\[
 m_C(P)=\operatorname{ord}_C(P\circ\pi),\qquad
 m_C(Q)=\operatorname{ord}_C(Q\circ\pi)
\]
are nonnegative integers and are not both zero. Since $\lambda$ is
irrational, this residue is nonzero. In local coordinates with $C=(x=0)$,
we therefore have
\[
 \pi^*\omega_0=\rho_C\frac{dx}{x}+\beta,
 \qquad \rho_C\neq0,
\]
with $\beta$ logarithmic along the remaining components. After clearing
the poles, the restriction of the defining form to $C$ is a nonzero
multiple of $dx$; hence $C$ is invariant. The same computation applies to
every exceptional component created later. Thus the resolution contains
no dicritical component, and the foliation is non-dicritical.
\end{remark}

Indeed, as already indicated in \cite{BerthierCerveau}, using the proof of Proposition~\ref{prop:local-simple-connectivity} one may state:
\begin{proposition}
\label{prop:local-simple-connectivitygerms}
Let $\omega\in\Lambda_m^1$, $m\geq3$, where
\[
 \omega=f_1\cdots f_n\sum_{j=1}^n\lambda_j\frac{df_j}{f_j},
\]
the $f_1,\ldots,f_n\in\Om$ are germs and $\lambda_j\in\C^*$. Assume that:
\begin{enumerate}[label=\textup{(\roman*)}]
\item $\omega$ is non-dicritical;
\item $\{\lambda_1,\ldots,\lambda_n\}$ satisfy a diophantine condition;
\item $\{\lambda_1,\ldots,\lambda_n,1\}$ are $\mathbb Z$-linearly independent;
\item $f_1,\ldots,f_n$ are in very general position in the sense of
Definition~\ref{def:very-general-position}.
\end{enumerate}
Then the leaves of the foliation $\mathcal F$ in
$U\setminus\operatorname{Sing}(\omega)$, not contained in
$\bigcup_{j=1}^n(f_j=0)$, are simply connected for a sufficiently small
neighborhood $U$ of the origin in $\C^m$.
\end{proposition}

\subsection{Dicriticalness}

Let $\mathcal F$ be a germ of holomorphic foliation of codimension one with
a singularity at $0\in\C^2$. The two-dimensional desingularization theorem
of Seidenberg~\cite{Seidenberg} ensures that there is a finite sequence of
quadratic blow-ups at singular points such that all the singularities in the
last step are reduced singularities. These are either nondegenerate
singularities, given by a germ of holomorphic vector field $X$ whose linear
part has two nonzero eigenvalues $\lambda,\mu$ with
$\lambda/\mu\notin\mathbb Q_+$, or saddle-nodes.
If we denote by $E$ the exceptional divisor produced by the blow-up sequence
above, the irreducible components of $E$ which are generically transverse to
the strict transform of $\mathcal F$ are called \emph{dicritical
components}. We say that $\mathcal F$ is \emph{non-dicritical} if $E$ has no
dicritical components; equivalently, $\mathcal F$ has only finitely many
separatrices at $0$.

For a germ of codimension-one holomorphic foliation $\mathcal F$ at
$0\in\C^m$, $m\geq3$, a similar definition is adopted by
Cano--Cerveau~\cite{CanoCerveau}; however, it involves the notion of
permissible blow-up, for which we refer to~\cite{CanoCerveau}. Roughly
speaking, we can say that $\mathcal F$ is dicritical if a certain
non-degenerate two-dimensional section of $\mathcal F$ admits infinitely
many separatrices.

\subsection{Nodal singularities}

Let $\omega(x,y)=A(x,y)\,dx+B(x,y)\,dy$ be a germ of holomorphic $1$-form
at the origin $0\in\C^2$. We shall say that $\omega=0$ is a
\emph{nodal-type singularity}, or a \emph{node}, if
\[
 X(x,y)=-B(x,y)\frac{\partial}{\partial x}
        +A(x,y)\frac{\partial}{\partial y}
\]
has a nonsingular linear part $DX(0,0)$ with eigenvalues $\lambda,\mu$
satisfying $\lambda/\mu\in\mathbb R_+$. A node separates the dynamics of a
foliation~\cite{CamachoRosas}.

We now introduce a similar notion for integrable $1$-forms in dimension
$m\geq2$. Consider a logarithmic singular $1$-form at $0\in\C^m$,
\[
 \omega_0=\sum_{j=1}^n\lambda_j\frac{df_j}{f_j}.
\]
We say that $\omega_0$ is not a node if
\[
 \frac{\sum_{j=1}^n\lambda_j}{\lambda_i}\notin\mathbb R_+,
 \qquad \forall i\in\{1,\ldots,n\}.
\]
We also say that $\omega_0$ is generic, or non-nodal, if:
\begin{enumerate}[label=\textup{(\roman*)}]
\item $(df_i\wedge df_j)(0)\neq0$, for all $i\neq j$;
\item $\omega_0$ is not a node, i.e.
$({\sum_{j=1}^n\lambda_j})/{\lambda_i}\notin\mathbb R_+$ for all $i$.
\end{enumerate}

For a generic non-nodal logarithmic $1$-form germ
$\omega_0=\sum_{j=1}^n\lambda_jdf_j/f_j$, consider the codimension-one
foliation $\mathcal F_{\omega_0}$ given by
\[
 \omega=0,\qquad
 \omega=\left(\prod_{i=1}^n f_i\right)
 \left(\sum_{j=1}^n\lambda_j\frac{df_j}{f_j}\right).
\]

We shall need the following.

\begin{lemma}[\cite{IlyashenkoYakovenko}]
Given a germ $\omega(x,y)=A(x,y)\,dx+B(x,y)\,dy$ such that
$X(x,y)=(-B,A)$ has eigenvalues $\lambda,\mu\neq0$ with
$\lambda/\mu\notin\mathbb R_+$, given a separatrix $\Lambda$ of
$\mathcal F_\omega$ and a transverse disk $\Sigma$, with
$\Sigma\cap\Lambda=\{p\}\neq\{0\}$, the saturation
$\operatorname{Sat}(\mathcal F_\omega,\Sigma)$ contains a neighborhood of
the origin of $\C^2$.
\end{lemma}

As a consequence we have:

\begin{lemma}[non-nodal
saturation lemma]
\label{lemma:nonnodalsaturation}
Let $\mathcal F$ be a germ of non-dicritical holomorphic foliation at
$0\in\C^2$. Assume that its reduction contains neither saddle-nodes nor
nodal singularities. Then, given a separatrix $\Lambda$ of $\mathcal F$
and a transverse disk $\Sigma$, with $\Sigma\cap\Lambda=\{p\}\neq\{0\}$,
the saturation $\operatorname{Sat}(\mathcal F,\Sigma)$ contains a
neighborhood of the origin $0\in\C^2$.
\end{lemma}

\begin{proof}
We use induction on the number $r$ of blow-ups in the minimal resolution of
$\mathcal F$~\cite{Seidenberg,CamachoSad}. For $r=0$, the singularity is
already irreducible and we invoke the lemma above. Suppose the result holds
for foliations resolvable by at most $r$ blow-ups, and assume that
$\mathcal F$ admits a resolution with $r+1$ blow-ups. Perform a first
blow-up
\[
 \pi_1:(\widetilde{\C^2},E)\longrightarrow(\C^2,0),
\]
where $E=\pi_1^{-1}(0)\simeq\mathbb P^1$ is the invariant exceptional
divisor. Put $\widetilde{\mathcal F}=\pi_1^*(\mathcal F)$ and denote by
$\{\widetilde q_1,\ldots,\widetilde q_s\}$ its singular set. Denote by
$\widetilde\Lambda$ the strict transform of $\Lambda$. Let also
$\widetilde\Sigma=\pi_1^{-1}(\Sigma)$; then
\[
 \widetilde\Sigma\cap\widetilde\Lambda=\{\widetilde p\},
 \qquad \widetilde p\notin\operatorname{Sing}(\widetilde{\mathcal F}).
\]
Applying the induction hypothesis at $\widetilde q_1$, we conclude that
$\operatorname{Sat}(\widetilde{\mathcal F},\widetilde\Sigma)$ contains a
neighborhood $V_{\widetilde q_1}$ of $\widetilde q_1$. Choose disks
$\widetilde\Sigma_j$ transverse to $E$, with
\[
 \widetilde\Sigma_j\cap E=\{\widetilde p_j\},\qquad j=1,\ldots,s,
\]
where $\widetilde p_j$ is nonsingular and sufficiently close to
$\widetilde q_j$.
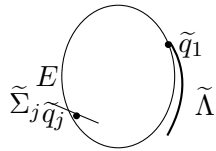
\begin{center}
\begin{tikzpicture}[scale=.65]
 \draw (0,0) ellipse (1.15 and 1.45); \node at (-1.45,0) {$E$};
 \fill (1.02,.65) circle (2pt) node[right] {$\widetilde q_1$};
 \fill (-.85,-.8) circle (2pt) node[left] {$\widetilde q_j$};
 \draw[thick] (1,-1.2) .. controls (1.45,-.4) and (1.35,.25) .. (1.05,.7);
 \node[right] at (1.35,-.5) {$\widetilde\Lambda$};
 \draw (-1.35,-.55)--(-.4,-.95); \node[left] at (-1.35,-.55) {$\widetilde\Sigma_j$};
\end{tikzpicture}
\end{center}
By induction, $\operatorname{Sat}(\widetilde{\mathcal F},\widetilde\Sigma_j)$
contains a neighborhood $V_{\widetilde q_j}$ of $\widetilde q_j$, for
$j=2,\ldots,s$. Using the invariance of $E$, we conclude that
$\operatorname{Sat}(\widetilde{\mathcal F},\widetilde\Sigma)$ contains a
neighborhood $\widetilde V$ of $E$ in $\widetilde{\C^2}$; projecting by
$\pi_1:\widetilde{\C^2}\to\C^2$, we conclude.
\end{proof}

\section{\texorpdfstring{Relative cohomology for rational logarithmic $1$-forms of two terms}{Relative cohomology for rational
logarithmic 1-forms}}
In this section we investigate the relative cohomology for rational  logarithmic 1-forms of
type $\omega_0=dQ/Q-\lambda\,dP/P$ where $P, Q$ are polynomials in $m$ complex variables. 
\begin{proof}[Proof of Theorem E]
Since clearly \textup{(2)} implies \textup{(1)}, we shall prove that
\textup{(1)} implies \textup{(2)}. Take an affine intersection point
$p_0\in\C^m$,
$P(p_0)=Q(p_0)=0$. We may
choose coordinates $(x,y,z_3,\ldots,z_m)$ centered at $p_0$ such that
$P=x$ and $Q=y$. Let
\[
 S=(z_3=\cdots=z_m=0)\simeq(\C^2,0).
\]
The restriction $\eta|_S$ has at most simple poles along $(xy=0)$.
Theorem~A gives meromorphic functions $a_0,h_0$ on $S$ and constants
$\alpha,\beta\in\C$ such that
\[
 \eta|_S=a_0\omega_0|_S+dh_0
 +\alpha\frac{dx}{x}+\beta\frac{dy}{y}.
\]
Condition \textup{(1)} and Theorem~B extend this decomposition to a
neighborhood $p_0\in W\subset\C^m$. Thus there are
$a,h\in\mathcal M(W)$ such that
\[
 \left(\left.\eta\right|_W-\alpha\frac{dP}{P}
 -\beta\frac{dQ}{Q}\right)\wedge\omega_0=dh\wedge\omega_0
 \quad\text{in }W.
\]
The following picture represents the transverse section $S$.

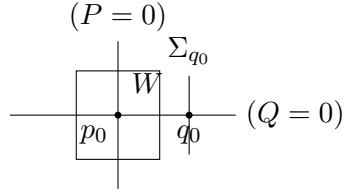
\begin{center}
\begin{tikzpicture}[scale=.65]
 \draw (-2.2,0)--(2.4,0) node[right] {$(Q=0)$};
 \draw (0,-1.5)--(0,1.5) node[above] {$(P=0)$};
 \draw (-.85,-.9) rectangle (.85,.9); \node at (.6,.65) {$W$};
 \fill (0,0) circle (2pt) node[below left] {$p_0$};
 \draw (1.45,-.8)--(1.45,.8) node[above] {$\Sigma_{q_0}$};
 \fill (1.45,0) circle (2pt) node[below] {$q_0$};
\end{tikzpicture}
\end{center}

Fix a nonsingular point $q_0\in(W\setminus\{p_0\})\cap(Q=0)$ and a
transverse disk $\Sigma_{q_0}$ centered at $q_0$. We now extend $h$ to a
meromorphic function in a neighborhood of the projective closure
$\overline{(Q=0)}\subset\mathbb{CP}^m$.

Let
\[
 Z=\overline{(P=0)}\cap\overline{(Q=0)}\cap H_\infty.
\]
By the simple normal-crossings hypothesis, $Z$ has codimension three in
$\mathbb{CP}^m$ when it is nonempty. We first work away from $Z$.

The required propagation of transversals follows from the non-nodal
saturation lemma applied on two-dimensional transverse sections. Indeed,
at every affine crossing of $(P=0)$ and $(Q=0)$ the foliation is locally,
with the remaining coordinates as parameters, equivalent to
\[
 \frac{dy}{y}-\lambda\frac{dx}{x}=0.
\]
Since $\lambda\notin\mathbb R$, this reduced singularity is non-nodal and
has no nodal separator. The parameterized saturation argument therefore
carries a transversal from one regular part of the invariant divisor to a
full neighborhood of the crossing. Along the regular part of $(Q=0)$ the
same propagation follows from flow boxes. By chaining these neighborhoods
along the connected smooth locus of the irreducible hypersurface $(Q=0)$,
every sufficiently small nearby
leaf through $\Sigma_{q_0}$ meets any prescribed transversal
$\Sigma_{q_1}$ centered at a regular point of $(Q=0)$. More explicitly,
compactness on relatively compact subsets of
$\overline{(Q=0)}\setminus Z$ gives finite chains of flow boxes centered at
regular points and parameterized non-nodal saturation neighborhoods centered
at its pairwise crossings with $(P=0)$ or $H_\infty$. Such a chain joining
$q_0$ to $q_1$ gives the asserted propagation.

Given any nonsingular point $q_1\in L_Q$, where $L_Q$ is the leaf of
$\mathcal F_{\omega_0}$ contained in $(Q=0)$, fix a transverse disk
$\Sigma_{q_1}$ centered at $q_1$, transverse to $(Q=0)$, with
$\Sigma_{q_1}\cap(Q=0)=\{q_1\}$. Given a point $z\in\Sigma_{q_0}$ close
enough to $q_0$, the leaf $L_z$ through $z$ intersects $\Sigma_{q_1}$.
Choose a point $w\in L_z\cap\Sigma_{q_1}$ and a path
$\gamma_{z,w}:[0,1]\to L_z$ joining $z$ to $w$.

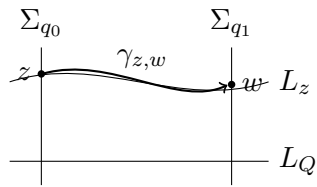
\begin{center}
\begin{tikzpicture}[scale=.7]
 \draw (-2.4,-.8)--(2.5,-.8) node[right] {$L_Q$};
 \draw (-2.4,.7).. controls (-.7,1.25) and (.8,.15)..(2.5,.7) node[right] {$L_z$};
 \draw (-1.8,-1.25)--(-1.8,1.35) node[above] {$\Sigma_{q_0}$};
 \draw (1.8,-1.25)--(1.8,1.35) node[above] {$\Sigma_{q_1}$};
 \fill (-1.8,.85) circle (2pt) node[left] {$z$};
 \fill (1.8,.65) circle (2pt) node[right] {$w$};
 \draw[->,thick] (-1.75,.86).. controls (-.5,1.2) and (.8,.2)..(1.72,.64)
 node[midway,above] {$\gamma_{z,w}$};
\end{tikzpicture}
\end{center}

Put
\[
 \theta:=\eta-\alpha\frac{dP}{P}-\beta\frac{dQ}{Q}.
\]
Then $d\theta=d\eta$, so that $d\theta\wedge\omega_0=0$. Therefore
$\theta$ is closed on each leaf of $\mathcal F_{\omega_0}$ outside its
polar set $(\theta)_\infty=(P=0)\cup(Q=0)$. In particular, the line
integral
\[
 \int_{\gamma_{z,w}}\theta
\]
depends only on the homotopy class of $\gamma_{z,w}$ with fixed extreme
points on the leaf $L_z$. By hypothesis, $L_z$ is simply connected;
therefore this integral does not depend on the choice of the path joining
$z$ to $w$.

Let now $z'\in\Sigma_{q_0}$ be another point such that $L_z=L_{z'}$ in
$W$. We have paths $\gamma_{z,w}$ and $\gamma_{z',w}$ in the same leaf,
joining $z$ and $z'$ to $w$, respectively. We claim that
\[
 \int_{\gamma_{z,w}}\theta=\int_{\gamma_{z',w}}\theta.
\]
Choose a path $\delta_{z,z'}:[0,1]\to L_z\cap W$ joining $z$ to $z'$.
The paths $\gamma_{z',w}*\delta_{z,z'}$ and $\gamma_{z,w}$ are both
contained in $L_z$, join $z$ to $w$, and are homotopic with fixed extreme
points. Hence
\[
 \int_{\gamma_{z',w}}\theta+\int_{\delta_{z,z'}}\theta
 =\int_{\gamma_{z,w}}\theta.
\]
Since $\delta_{z,z'}\subset L_z\cap W$, we have
\[
 \int_{\delta_{z,z'}}\theta
 =\int_{\delta_{z,z'}}(a\omega_0+dh)=h(z')-h(z).
\]
Thus
\[
 \int_{\gamma_{z',w}}\theta+h(z')
 =\int_{\gamma_{z,w}}\theta+h(z).
\]

Let us now consider another base point $q'\in L_Q\cap W$ and a transverse
disk $\Sigma_{q'}$, centered at $q'$. Given $w\in\Sigma_{q_1}$ and points
$z\in\Sigma_{q_0}$, $z'\in\Sigma_{q'}$ such that
$L_z=L_{z'}\ni w$, consider paths $\gamma_{z,w}$ and $\gamma_{z',w}$ in
$L_z=L_{z'}$, joining $z$ and $z'$ to $w$, respectively. Then
\[
 h(z)+\int_{\gamma_{z,w}}\theta
 =h(z')+\int_{\gamma_{z',w}}\theta.
\]
Indeed, choose a path $\delta_{z',z}:[0,1]\to L_z\cap W$ joining $z'$ to
$z$. By the simple connectivity of the open leaves, the line integrals on
$\gamma_{z',w}$ and $\gamma_{z,w}*\delta_{z',z}$ coincide, and
\[
 \int_{\delta_{z',z}}\theta=h(z)-h(z').
\]

This shows that we may define a function $\widetilde h$ by setting, for
$w\in L_z\cap\Sigma_{q_1}$,
\[
 \widetilde h(w):=h(z)+\int_{\gamma_{z,w}}\theta,
\]
and this does not depend on the choices of $z$, $q$ and $\gamma_{z,w}$.
By construction, on each open leaf $L_z$ of $\mathcal F_{\omega_0}$ we
have
\[
 d\widetilde h|_{L_z}=\theta|_{L_z},
 \qquad d\widetilde h\wedge\omega_0=\theta\wedge\omega_0.
\]
The holomorphic dependence of solutions of the leafwise primitive equation
in foliation flow boxes shows that $\widetilde h$ is holomorphic away from
the polar divisor and depends holomorphically on the transverse parameters.

We now show that $\widetilde h$ is a meromorphic function. In suitable
coordinates $(x,y,z)$, with $z=(z_3,\ldots,z_m)$, we have
\[
 \omega_0=\frac{dy}{y}-\lambda\frac{dx}{x},
 \qquad
 \widetilde h(x,y,z)=
 \sum_{i,j\in\mathbb Z}\widetilde h_{ij}(z)x^iy^j.
\]
Restricting to the transverse bidiscs $(z=\text{constant})$ gives
\[
 xy\,d\widetilde h\wedge\omega_0\big|_{z=\mathrm{const.}}
 =\sum_{i,j\in\mathbb Z}\widetilde h_{ij}(z)
 (i+\lambda j)x^iy^j\,dx\wedge dy.
\]
Since $xy\,d\widetilde h\wedge\omega_0
=xy\,\theta\wedge\omega_0$ is meromorphic in a neighborhood of
$0\in\C^m$, there is, after shrinking the parameter polydisc, an
$N\geq0$ such that its Laurent coefficients vanish
whenever $i<-N$ or $j<-N$. For those indices,
\[
 (i+\lambda j)\widetilde h_{ij}(z)=0.
\]
Because $\lambda\notin\mathbb Q$, $i+\lambda j=0$ only for $(i,j)=(0,0)$.
Hence $\widetilde h_{ij}(z)=0$ whenever $i<-N$ or $j<-N$: the series has a
finite principal part. For the remaining coefficients, the integer-index
diophantine estimate bounds $|i+\lambda j|^{-1}$ by a polynomial in
$|i|+|j|$. Multiplication of convergent Laurent coefficients by such a
polynomial preserves convergence uniformly on every smaller parameterized
bidisc. Thus $\widetilde h(x,y,z)$ extends meromorphically to $W$. Therefore
$\widetilde h$ is already defined meromorphically in a neighborhood $U$ of
$L_Q$.

Now we show that $h$ extends to a neighborhood of the singular points
$p_1\in(Q=0)\cap\operatorname{Sing}(\mathcal F_{\omega_0})$. For such a
singular point, choose a neighborhood $W_1$ of $p_1$ and a meromorphic
function $h_1\in\mathcal M(W_1)$ such that on $W_1$ we have
\[
 \left(\eta-\alpha_1\frac{dP}{P}-\beta_1\frac{dQ}{Q}\right)
 \wedge\omega_0=dh_1\wedge\omega_0
\]
for some $\alpha_1,\beta_1\in\C$. On $W_1\cap W$ we have both this
identity and
\[
 \left(\eta-\alpha\frac{dP}{P}-\beta\frac{dQ}{Q}\right)
 \wedge\omega_0=dh\wedge\omega_0.
\]
Therefore
\[
 \omega_0\wedge\left((\alpha_1-\alpha)\frac{dP}{P}
 +(\beta_1-\beta)\frac{dQ}{Q}\right)
 =\omega_0\wedge d(h-h_1).
\]
By the Laurent-coefficient argument recorded in
Lemma~\ref{lem:log-rigidity} below, applied on transverse bidiscs with
the remaining coordinates as parameters, $h-h_1$ is independent of the
two transverse coordinates. The full differential identity then shows
that it is independent of the parameter coordinates as well. Hence
$h-h_1$ is constant, and
$h=(h-h_1)+h_1$ extends meromorphically to $W_1$.

We prove now that $h$ extends to the points of
$\overline{(Q=0)}\cap H_\infty$, where
$H_\infty=\mathbb{CP}^m\setminus\C^m$ is the hyperplane at infinity.
Put $p=\deg P$ and $q=\deg Q$. Since $\lambda\notin\mathbb R$, the residue
$-(q-\lambda p)$ along $H_\infty$ is nonzero. Thus $H_\infty$ is invariant.
Every point of $H_\infty\cap\overline{(Q=0)}$ outside $Z$ is a reduced
non-nodal crossing of invariant components. Let $q_\infty$ be such a
point. In local coordinates $(u,v,z_3,\ldots,z_m)$ centered at
$q_\infty$, we have
\[
 H_\infty=(u=0),\qquad \overline{(Q=0)}=(v=0).
\]
Writing the affine polynomials in this projective chart gives
\[
 P=u^{-p}\widetilde P(u,v,z),\qquad
 Q=u^{-q}\widetilde Q(u,v,z),
\]
where $\widetilde P,\widetilde Q$ are the local equations of the projective
closures. Therefore
\[
 \omega_0(u,v,z)
 =\frac{d\widetilde Q}{\widetilde Q}
 -\lambda\frac{d\widetilde P}{\widetilde P}
 -(q-\lambda p)\frac{du}{u}.
\]
Since $q_\infty\notin Z$, it does not belong to $\overline{(P=0)}$.
Thus $\widetilde P(q_\infty)\neq0$. By the normal-crossings
hypothesis, after a local coordinate change we may write
$\widetilde Q=vU(u,v,z)$, where $U(q_\infty)\neq0$. We obtain
\[
 \frac{d\widetilde Q}{\widetilde Q}
 -\lambda\frac{d\widetilde P}{\widetilde P}
 =\frac{dv}{v}+d\psi,
\]
where $\psi$ is holomorphic. Hence
\[
 \omega_0(u,v,z)=\frac{dv}{v}-(q-\lambda p)\frac{du}{u}+d\psi.
\]
After replacing $v$ by $ve^{\psi}$, we may assume
\[
 \omega_0=\frac{dv}{v}-(q-\lambda p)\frac{du}{u}.
\]
Since $\lambda\notin\mathbb R$, $q-\lambda p\notin\mathbb R$; the
integer-index estimate is also inherited by the pair
$(q-\lambda p,1)$, because
\[
 i(q-\lambda p)+j=(-ip)\lambda+(iq+j),
\]
and the norm of the integer coefficient vector on the right is bounded by
a constant multiple of $|i|+|j|$. The preceding arguments therefore show
that $h$ extends to a neighborhood of $q_\infty$.

We have thus obtained a meromorphic function $h$ on a neighborhood of
$\overline{(Q=0)}\setminus Z$ in $\mathbb{CP}^m$ such that
\[
 \left(\eta-\alpha\frac{dP}{P}-\beta\frac{dQ}{Q}\right)\wedge\omega_0
 =dh\wedge\omega_0.
\]
It remains to extend $h$ across $Z$. Let $r\in Z$. Since the projective
divisor has simple normal crossings, there are local coordinates
$(u,v,w,z)$ centered at $r$, where $z$ denotes the remaining coordinates,
such that
\[
 H_\infty=(u=0),\qquad
 \overline{(Q=0)}=(v=0),\qquad
 \overline{(P=0)}=(w=0).
\]
In these coordinates, $Z=(u=v=w=0)$ and, inside
$\overline{(Q=0)}=(v=0)$, it is given by $(u=w=0)$. Choose sufficiently
small polydiscs in the variables $(u,w)$ and $(v,z)$. The neighborhood of
$\overline{(Q=0)}\setminus Z$ on which $h$ has already been constructed
contains, after shrinking the parameter polydisc, a product Hartogs figure:
its shell lies in the two variables $(u,w)$, while $(v,z)$ are parameters.
The parameterized Hartogs extension theorem for meromorphic functions
therefore extends $h$ meromorphically to a full neighborhood of $r$.
On the original Hartogs figure we have
\[
 dh\wedge\omega_0
 =\left(\eta-\alpha\frac{dP}{P}
 -\beta\frac{dQ}{Q}\right)\wedge\omega_0.
\]
Both sides are meromorphic on the full neighborhood, so the same identity
holds there by the meromorphic identity principle.
Since $r\in Z$ was arbitrary, $h$ is meromorphic on a neighborhood
$\widetilde W$ of the whole projective hypersurface
$\overline{(Q=0)}$.

By the Levi extension theorem~\cite[Chapter~1, p.~13]{siu}, in the form used by
Camacho--Lins Neto--Sad~\cite[Lemma~5, Section~3, p.~437]{CamachoLinsNetoSad},
$h$ extends meromorphically to $\mathbb{CP}^m$; this extension is rational.
On the dense open set where $\omega_0$ is regular, the identity above
shows that the meromorphic $1$-form
\[
 \eta-\alpha\frac{dP}{P}-\beta\frac{dQ}{Q}-dh
\]
is proportional to $\omega_0$. The proportionality factor extends
meromorphically to $\mathbb{CP}^m$ and is therefore a rational function
$a$. Restricting to the affine part gives
\[
 \eta=a\omega_0+dh+
 \alpha\frac{dP}{P}+\beta\frac{dQ}{Q},
\]
which proves \textup{(2)}.
\end{proof}

\begin{lemma}\label{lem:log-rigidity}
Let
\[
 \omega_0=\frac{dy}{y}-\lambda\frac{dx}{x},\qquad
 \lambda\in\C\setminus\mathbb Q,
\]
let $\alpha,\beta\in\C$, and let
$H\in\mathcal M(U\setminus\{x=0\})$, where
$0\in U\subset\C^2$ is a bidisk centered at the origin. Assume that the
polar divisor of $H$ in $U\setminus\{x=0\}$ is supported on $(y=0)$ and
has finite order there. Assume also that
\begin{equation}
 \label{eq:log-rigidity}
 \omega_0\wedge\left(\alpha\frac{dx}{x}+\beta\frac{dy}{y}\right)
 =\omega_0\wedge dH
 \quad\text{in }U\setminus\{x=0\}.
\end{equation}
Then $H$ is constant in $U\setminus\{x=0\}$.
\end{lemma}

\begin{proof}
Write the Laurent series
\[
 H(x,y)=\sum_{\substack{i\in\mathbb Z\\j\geq j_0}}H_{ij}x^iy^j.
\]
We first claim that $H$ is meromorphic in $U$. Indeed,
\begin{align*}
 xy\,\omega_0\wedge dH
 &=(x\,dy-\lambda y\,dx)\wedge(H_x\,dx+H_y\,dy)\\
 &=-[xH_x+\lambda yH_y]\,dx\wedge dy\\
 &=-\left[\sum(i+\lambda j)H_{ij}x^iy^j\right]dx\wedge dy.
\end{align*}
By~\eqref{eq:log-rigidity}, $xy\,\omega_0\wedge dH$ is meromorphic in $U$. Hence
$(i+\lambda j)H_{ij}=0$ for the sufficiently negative terms. Since
$\lambda\notin\mathbb Q$, the corresponding $H_{ij}$ vanish, and $H$ is
meromorphic in $U$.

We now claim that $H$ is constant. Multiplying~\eqref{eq:log-rigidity} by $xy$ gives
\[
 -(\alpha+\lambda\beta)
 =-\sum_{\substack{i\geq i_0\\j\geq j_0}}
 (i+\lambda j)H_{ij}x^iy^j.
\]
The constant coefficient on the right is zero. Hence the assumed identity
first gives $\alpha+\lambda\beta=0$. It then gives $H_{ij}=0$ for every
$(i,j)\neq(0,0)$, and $H$ is constant.
\end{proof}

\begin{lemma}
\label{lemma:simplyconnectedreduced}
Let $P,Q\in\C[z_1,z_2]$ be homogeneous irreducible polynomials such
that:
\begin{enumerate}[label=\textup{(\roman*)}]
\item the curves $(P=0),(Q=0)\subset\C^2$ are smooth and transverse;
\item $\lambda\in\C\setminus\mathbb Q$ satisfies a diophantine
condition.
\end{enumerate}
Then the affine open leaves of the foliation $\mathcal F_{\omega_0}$,
where
\[
 \omega_0=\frac{dQ}{Q}-\lambda\frac{dP}{P},
\]
are simply connected in $\C^2$.
\end{lemma}

\begin{proof}
An irreducible homogeneous polynomial in two variables over $\C$ is linear.
Since $(P=0)$ and $(Q=0)$ are transverse, a linear change of coordinates
reduces the foliation to
\[
 \omega_0=\frac{dy}{y}-\lambda\frac{dx}{x}.
\]
It is enough to prove the following local claim and then use dilation.

\emph{Claim.} For a sufficiently small neighborhood $U$ of
$0\in\C^2$, the leaves of $\mathcal F_{\omega_0}|_U$ are
simply connected.

The open leaf through $(x_0,y_0)\in(\C^*)^2$ is parametrized by
\[
 \Phi(s)=(x_0e^s,y_0e^{\lambda s}).
\]
If $\Phi(s)=\Phi(s')$, then $s-s'=2\pi i k$ and
$\lambda(s-s')=2\pi i\ell$ for some $k,\ell\in\mathbb Z$. Since
$\lambda\notin\mathbb Q$, this forces $k=\ell=0$. Thus $\Phi$ is injective.
In a sufficiently small bidisc the parameter domain is the intersection of
the half-planes
\[
 \operatorname{Re}s<c_1,\qquad
 \operatorname{Re}(\lambda s)<c_2,
\]
and is therefore convex. The local open leaves are simply connected.

Finally, $\delta_t^*\omega_0=\omega_0$ and
$\delta_t(L_z)=L_{tz}$. Given a loop in a global open leaf, dilate it into
the small bidisc, contract it in the corresponding local leaf, and apply
the inverse dilation. This proves the global assertion.
\end{proof}

\section{\texorpdfstring{Relative cohomology for germs of logarithmic
$1$-forms}{Relative cohomology for germs of logarithmic 1-forms}}

We fix some terminology used in this section. A \emph{meromorphic first
integral} of a foliation defined by $\omega$ is a meromorphic function $H$
such that
\[
 dH\wedge\omega=0.
\]
Given a meromorphic $1$-form $\Theta$, a \emph{relative primitive} is a
meromorphic function $h$ satisfying
\[
 (\Theta-dh)\wedge\omega=0;
\]
equivalently, $\Theta=a\omega+dh$ for a meromorphic function $a$. In
normal-crossings coordinates, a Laurent monomial $x^k$ is called
\emph{resonant} when its exponent vector is parallel to the residue vector,
that is, when
\[
 k_i\lambda_j-k_j\lambda_i=0
 \qquad\text{for all }i,j.
\]
Finally, the \emph{logarithmic obstruction} is the collection of residue
coefficients of $\Theta$ along the irreducible components of the polar
divisor which cannot be absorbed into a multiple of $\omega$ or an exact
meromorphic term.

\begin{lemma}[Vanishing of the additive holonomy cocycle]
Let $\mathcal G$ be a finitely generated group of germs of transverse
holonomy maps which, in a coordinate $z$, are simultaneously of the form
\[
 g(z)=\mu_gz,
 \qquad \mu_g\in\C^*.
\]
Assume either that $\mathcal G$ is finite, or that it contains an element
$g_0(z)=\mu z$ with $\mu$ not a root of unity. In the latter case, if
$|\mu|=1$, assume that there are $C,\tau>0$ such that
\[
 |\mu^n-1|\geq\frac{C}{|n|^\tau},
 \qquad n\in\mathbb Z\setminus\{0\}.
\]
Let $c_g$ be meromorphic germs of uniformly finite principal part which
satisfy
\[
 c_{g\circ f}=c_f+c_g\circ f,
\]
and whose constant Laurent coefficient is zero.
Then there is a meromorphic germ $b$ such that
\[
 c_g=b\circ g-b,
 \qquad g\in\mathcal G.
\]
If all the $c_g$ are holomorphic, then $b$ is holomorphic.
\end{lemma}

\begin{proof}
If $\mathcal G$ is finite, put
\[
 b:=-\frac1{|\mathcal G|}\sum_{f\in\mathcal G}c_f.
\]
The cocycle identity and reindexing the finite sum give
$b\circ g-b=c_g$ for every $g\in\mathcal G$.

Assume now that $\mathcal G$ is infinite and choose $g_0$ as in the
statement. Write
\[
 c_{g_0}(z)=\sum_{n\geq-N}c_nz^n.
\]
By hypothesis, $c_0=0$. Define
\[
 b(z)=\sum_{\substack{n\geq-N\\n\neq0}}
 \frac{c_n}{\mu^n-1}z^n.
\]
If $|\mu|\neq1$, the denominators are bounded away from zero for one of
the two tails and grow exponentially for the other; since the negative
tail is finite, the series converges. If $|\mu|=1$, the stated estimate
allows only a polynomial loss in the coefficients, and the series again
converges. Thus $b$ is meromorphic, with no pole when $c_{g_0}$ is
holomorphic, and
\[
 c_{g_0}=b\circ g_0-b.
\]

For $g\in\mathcal G$, put
\[
 d_g:=c_g-(b\circ g-b).
\]
The family $(d_g)$ is again an additive cocycle and $d_{g_0}=0$. Since the
multiplicative maps commute, the cocycle identity applied to
$g\circ g_0=g_0\circ g$ gives
\[
 d_g\circ g_0=d_g.
\]
A meromorphic germ invariant under $z\mapsto\mu z$, with $\mu$ not a root
of unity, is constant. Its constant Laurent coefficient is zero, so
$d_g=0$. Hence
$c_g=b\circ g-b$ for every $g\in\mathcal G$.
\end{proof}

\begin{lemma}[Leafwise gluing near a resolved divisor]
Let $\mathcal F$ be a non-dicritical foliation germ at $0\in\C^2$ whose
reduction contains neither saddle-nodes nor nodal singularities, and let
\[
 \pi:(M,E)\longrightarrow(\C^2,0)
\]
be a reduction with total invariant divisor $\widetilde D$. Let
$\widetilde\omega$ define the lifted foliation and let $\Theta$ be a
meromorphic $1$-form with finite polar order along $\widetilde D$. Assume
that:
\begin{enumerate}[label=\textup{(\roman*)}]
 \item $d\Theta\wedge\widetilde\omega=0$;
 \item every open leaf in a sufficiently small punctured neighborhood is
 simply connected;
 \item at every point of $\widetilde D$ there are local meromorphic
 functions $a_V,h_V$, of finite pole order, such that
 \[
 \Theta|_V=a_V\widetilde\omega+dh_V.
 \]
 \item on a transverse disk to each invariant component, the holonomy is
 either finite or simultaneously linearizable with multipliers satisfying
 the estimate in the additive holonomy lemma, and the additive cocycle
 determined by the local primitives has zero constant Laurent coefficient.
\end{enumerate}
Then there are meromorphic functions $\widetilde a,\widetilde h$, defined
on a neighborhood of $E$, such that
\[
 \Theta=\widetilde a\widetilde\omega+d\widetilde h.
\]
\end{lemma}

\begin{proof}
Choose a separatrix $\Lambda$ and a small transverse disk $\Sigma$ meeting
$\Lambda$ at a regular point. A local decomposition supplies a meromorphic
primitive $h_\Sigma$ on a flow box containing $\Sigma$. Since the reduction
has no nodal separator, the non-nodal saturation lemma (Lemma~\ref{lemma:nonnodalsaturation}) shows that the
saturation of this flow box contains a punctured neighborhood of the
origin; after lifting, it contains a punctured neighborhood of the total
exceptional divisor.

For a holonomy germ $g$ represented by a family of leafwise paths
$\gamma_{g,z}$ from $z$ to $g(z)$, set
\[
 c_g(z):=h_\Sigma(g(z))-h_\Sigma(z)
 -\int_{\gamma_{g,z}}\Theta.
\]
The family $(c_g)$ satisfies
\[
 c_{g\circ f}=c_f+c_g\circ f.
\]
By \textup{(iv)} and the additive holonomy lemma there is a meromorphic
germ $b$ such that $c_g=b\circ g-b$. Replacing $h_\Sigma$ by
$h_\Sigma-b$ makes every $c_g$ vanish. Thus the initial primitive is
compatible with all returns of the holonomy pseudogroup.

Continue $h_\Sigma$ through this saturation by integrating $\Theta$ along
the leaves. If two chains of flow boxes reach the same plaque, their
concatenation gives a loop in the corresponding open leaf. The restriction
of $\Theta$ to that leaf is closed, and the leaf is simply connected, so
the two continuations agree. The parameterized monodromy argument in flow
boxes also shows that the resulting function $h^\circ$ is holomorphic off
$\widetilde D$ and satisfies
\[
 dh^\circ\wedge\widetilde\omega
 =\Theta\wedge\widetilde\omega.
\]

Near a smooth point of $\widetilde D$, the meromorphic flow-box lemma
extends $h^\circ$ with finite pole order. Near a corner, compare
$h^\circ$ with a local meromorphic primitive $h_V$ supplied by
\textup{(iii)}. Their difference is a single-valued first integral on the
punctured chart. Its germ on the incoming saturated flow box is
meromorphic; analytic continuation through the non-nodal saturation, or
equivalently its resonant Laurent expansion, shows that it extends
meromorphically across the corner. Hence $h^\circ$ extends to a meromorphic
function $\widetilde h$ on a neighborhood of $E$.

The form $\Theta-d\widetilde h$ annihilates the tangent line field of the
lifted foliation. It is therefore a meromorphic multiple
$\widetilde a\widetilde\omega$. This proves the lemma.
\end{proof}

\begin{proof}[Proof of Theorem C]
Let
\[
 \pi:(M,E)\longrightarrow(U,0)
\]
be an embedded resolution of the reduced divisor $D$, chosen simultaneously
with the reduction of $\mathcal F_{\omega_0}$. Denote by
$\widetilde D=\pi^{-1}(D)$ the total transform and put
\[
 \widetilde\omega_0=\pi^*\omega_0,
 \qquad \widetilde\eta=\pi^*\eta.
\]
The divisor $\widetilde D$ has normal crossings and its exceptional dual
graph is a tree.

If $C$ is an irreducible component of $\widetilde D$, then the residue of
$\widetilde\omega_0$ along $C$ is
\[
 \rho_C=\sum_{j=1}^n m_{Cj}\lambda_j,
 \qquad m_{Cj}=\operatorname{ord}_C(f_j\circ\pi)\in\mathbb Z_{\geq0}.
\]
The vector $(m_{C1},\ldots,m_{Cn})$ is nonzero. Hence $\rho_C\neq0$ by
$\mathbb Z$-linear independence. At a corner $C\cap C'$, suitable local
coordinates $(x,y)$ give
\[
 \widetilde\omega_0
 =\rho_C\frac{dx}{x}+\rho_{C'}\frac{dy}{y}+d\varphi,
\]
where $\varphi$ is holomorphic. Absorbing $d\varphi$ into one of the
coordinates reduces the local equation to a linear logarithmic form.
In particular, the lifted foliation has no saddle-node at a corner; at a
smooth point of the total transform it is regular because $\rho_C\neq0$.
Every nonzero divisor occurring in its Laurent coefficient equations has
the form
\[
 p\rho_{C'}-q\rho_C
 =\sum_{j=1}^n(pm_{C'j}-qm_{Cj})\lambda_j,
 \qquad (p,q)\in\mathbb Z^2.
\]
The integer-index diophantine condition therefore gives the required
polynomial lower bound. If this combination is zero, the corresponding
Laurent monomial is a meromorphic first integral; we normalize its
coefficient in the primitive to be zero. Thus the local coefficient
calculation gives convergent meromorphic solutions and finite principal
parts at every corner. The nonsingular meromorphic decomposition lemma
gives the same conclusion at regular points of $\widetilde D$.

We next fix the logarithmic obstruction. At a generic smooth point of the
strict transform of $(f_j=0)$, the local two-dimensional calculation gives
a uniquely determined coefficient $\mu_j\in\mathbb C$ of
$d(f_j\circ\pi)/(f_j\circ\pi)$, modulo a multiple of
$\widetilde\omega_0$. Put
\[
 \Theta=\widetilde\eta-
 \sum_{j=1}^n\mu_j\frac{d(f_j\circ\pi)}{f_j\circ\pi}.
\]
We next show that this subtraction removes the logarithmic obstruction
along every component of the total transform. Let $C$ be an irreducible
exceptional component. At a generic point of $C=(x=0)$, write
\[
 f_j\circ\pi=u_jx^{m_{Cj}},
\]
where $u_j$ is a holomorphic unit. Consequently,
\[
 \frac{d(f_j\circ\pi)}{f_j\circ\pi}
 =m_{Cj}\frac{dx}{x}+\frac{du_j}{u_j},
\]
and the logarithmic coefficient induced by the correction along $C$ is
\[
 \mu_C=\sum_{j=1}^n m_{Cj}\mu_j.
\]
Equivalently, a positively oriented meridian $\gamma_C$ about $C$ projects
in $H_1(U\setminus D,\mathbb Z)$ to
\[
 \sum_{j=1}^n m_{Cj}\gamma_j,
\]
where $\gamma_j$ is a meridian about $(f_j=0)$, and hence
\[
 \frac{1}{2\pi i}\int_{\gamma_C}
 \sum_{j=1}^n\mu_j\frac{d(f_j\circ\pi)}{f_j\circ\pi}
 =\sum_{j=1}^n m_{Cj}\mu_j=\mu_C.
\]

The logarithmic obstruction is understood modulo
$\widetilde\omega_0$: changing the local coefficient $a_V$ changes the
residue along $C$ by a multiple of
$\rho_C=\operatorname{Res}_C(\widetilde\omega_0)$. Thus the relevant
obstruction is the residue class modulo $\mathbb C\rho_C$. Since the
obstruction on an exceptional meridian is the same integral combination
of the obstructions on the original meridians, the choice of the
coefficients $\mu_j$ gives
\[
 \operatorname{Res}_C(\Theta)\in\mathbb C\rho_C
\]
for every exceptional component $C$. This remaining residue can be
absorbed into $a_V\widetilde\omega_0$.

At a smooth point of the total transform, the meromorphic flow-box
decomposition then applies. At a corner $C\cap C'$, the residue vector of
$\Theta$, after subtracting a suitable local multiple of
$\widetilde\omega_0$, is zero. The local Laurent calculation therefore
leaves only a multiple of $\widetilde\omega_0$ and an exact meromorphic
term. Consequently, on every normal-crossings chart there are meromorphic
functions $a_V,h_V$ such that
\begin{equation}
 \Theta|_V=a_V\widetilde\omega_0+dh_V.
 \label{eq:local-decomposition-five}
\end{equation}

We also verify the holonomy condition in the gluing lemma. Since
$\widetilde\omega_0$ is closed logarithmic, the holonomy along every
invariant component is simultaneously linearizable in a logarithmic
transverse coordinate. Indeed, exponentiating a primitive of the closed
logarithmic form after division by the residue gives a local transverse
coordinate; changes between such coordinates are multiplicative constants.
Thus all holonomy generators are linear in the same coordinate. At a
corner $C\cap C'$, their multipliers have the form
\[
 \exp\left(2\pi i\frac{\rho_{C'}}{\rho_C}\right).
\]
If the multiplier group is finite, the finite case of the additive
holonomy lemma applies. Otherwise it contains a multiplier $\mu$ which is
not a root of unity. When $|\mu|\neq1$ no small-divisor estimate is needed.
When $|\mu|=1$, the lower bound for $|\mu^k-1|$ follows from the
integer-index diophantine estimate. Indeed, for each $k\neq0$ choose the
nearest integer $r$ to $k\rho_{C'}/\rho_C$. Then $r=O(|k|)$ and
\[
 \left|\exp\left(2\pi i k\frac{\rho_{C'}}{\rho_C}\right)-1\right|
 \asymp \frac{|k\rho_{C'}-r\rho_C|}{|\rho_C|}.
\]
Moreover,
\[
 k\rho_{C'}-r\rho_C
 =\sum_{j=1}^n
 (k m_{C'j}-r m_{Cj})\lambda_j,
\]
and the norm of its integer coefficient vector is $O(|k|)$. The
integer-index condition therefore gives the required polynomial lower
bound for $|\mu^k-1|$.
Finally, the constant Laurent coefficient of the additive cocycle is the
logarithmic period of $\Theta$ along the corresponding limiting cycle.
The choice of the coefficients $\mu_j$ makes all these periods zero,
including those of exceptional meridians by the relation above. Thus
condition \textup{(iv)} of the gluing lemma is satisfied.

There is also a common lower Laurent bound for the cocycle. Choose a finite
normal-crossings cover of the compact exceptional divisor. On each member,
$\Theta$ and the selected local primitive have finite pole order. The
maximum of these finitely many orders gives an integer $N$ valid throughout
the cover. Linear substitutions $z\mapsto\mu z$ do not change Laurent
exponents, and leafwise integration does not create lower exponents.
Therefore every cocycle germ $c_g$ has no term $z^k$ with $k<-N$, as
required by the additive holonomy lemma.

The hypotheses of the leafwise gluing lemma are now satisfied. Condition
\textup{(iii)} is precisely the family of local decompositions
\eqref{eq:local-decomposition-five}; condition \textup{(i)} follows from
$d\Theta\wedge\widetilde\omega_0=0$; condition \textup{(ii)} is the lifted
form of the assumed simple connectivity of the open leaves. Finally, the
new non-nodal-resolution hypothesis allows the use of the saturation
lemma; the preceding paragraph verifies condition \textup{(iv)}. We
therefore obtain meromorphic functions
$\widetilde a,\widetilde h$ on a neighborhood of the whole exceptional
divisor such that
\begin{equation}
 \widetilde\eta
 =\widetilde a\widetilde\omega_0+d\widetilde h
 +\sum_{j=1}^n\mu_j
   \frac{d(f_j\circ\pi)}{f_j\circ\pi}.
 \label{eq:global-decomposition-five}
\end{equation}

Finally, a meromorphic function on a neighborhood of the exceptional set of
a proper modification represents an element of the meromorphic function
field of $(U,0)$. Thus $\widetilde a$ and $\widetilde h$ descend to
meromorphic germs $a,h\in\mathcal M_2$.
Descending~\eqref{eq:global-decomposition-five} gives
\[
 \eta=a\omega_0+dh+
 \sum_{j=1}^n\mu_j\frac{df_j}{f_j},
\]
as required.
\end{proof}

\begin{remark}
A higher-dimensional version follows from
\hyperref[thm:B]{Theorem~B} whenever there is a two-dimensional section in
general position on which the hypotheses of Theorem~C hold for the
restricted logarithmic foliation. We do not claim that version here.
\end{remark}

\section{Analytic integrable deformations of logarithmic foliation germs}
We shall now apply our techniques above to the study of deformations of integrable 1-forms. 
\subsection{Deformations: the polynomial homogeneous case}

\begin{theorem}\label{thm:six}
Let $m\geq3$ and let $P,Q$ be reduced irreducible homogeneous polynomials
of the same degree and in general projective position.
Let $\lambda\in\C$.
Given an analytic deformation
\[
 \Omega_t=\Omega_0+\sum_{j=1}^{\infty}t^j\Omega_j
\]
by homogeneous polynomial integrable $1$-forms of the same degree as
\[
 \Omega_0=P\,dQ-\lambda Q\,dP,
\]
assume that both $(P=0)$ and $(Q=0)$ are invariant by $\Omega_t$ for every
$t$. Then
$\mathcal F_t:\Omega_t=0$ is logarithmic; indeed,
\[
 \Omega_t=PQ\left(a(t)\frac{dQ}{Q}
 +b(t)\frac{dP}{P}\right)
\]
for some holomorphic germs $a(t),b(t)\in\mathcal O_1$. Moreover,
$a(0)=1$ and $b(0)=-\lambda$.
\end{theorem}

\begin{proof}
For every fixed $t$, apply the homogeneous logarithmic divisor theorem to
$\Omega_t$. This gives constants $a(t),b(t)$ such that
\[
 \frac{\Omega_t}{PQ}=a(t)\frac{dQ}{Q}+b(t)\frac{dP}{P}.
\]
These constants are the residues of $\Omega_t/(PQ)$ along $(Q=0)$ and
$(P=0)$, respectively. Since $\Omega_t$ depends analytically on $t$, its
residues do as well. Hence $a(t),b(t)\in\mathcal O_1$.
\end{proof}

\begin{remark}
Without the invariance of $(PQ=0)$, the preceding logarithmic conclusion
does not follow in higher dimension merely from integrability. If $m\geq3$,
we cannot in general conclude that
\[
 \omega_j=a_j\omega_0+dh_j,
\]
because we do not necessarily have
\[
 d\omega_j\wedge\omega_0=0.
\]
\end{remark}

\subsection{Analytic deformations of simple germs}
Let us now prove Theorems~G and H. For this we shall need:
\begin{lemma}[Finite-pole relative decomposition]
\label{lem:finite-pole-decomposition}
Let $F,G\in\mathcal O_m$, $m\geq2$, be reduced irreducible germs in general
position at the origin in the sense of
Definition~\ref{def:two-component-general-position}, and put
\[
 \omega_0=\frac{dG}{G}-\lambda\frac{dF}{F}.
\]
Assume that $\lambda$ satisfies the integer-index diophantine condition.
If $m\geq3$, use the submersion theorem to choose local coordinates
$(x,y,z_3,\ldots,z_m)$ such that $F=x$ and $G=y$, and put
\[
 S=(z_3=\cdots=z_m=0).
\]
If $m=2$, put $S=(\C^2,0)$.
Let $\eta$ be a meromorphic $1$-form whose polar divisor is supported on
$(FG=0)$ and has arbitrary finite order. If
\[
 d\eta\wedge\omega_0=0,
\]
then there are meromorphic germs $a,h$, also of finite pole order along
$(FG=0)$, and constants $\alpha,\beta\in\C$ such that
\[
 \eta=a\omega_0+dh+\alpha\frac{dF}{F}
 +\beta\frac{dG}{G}.
\]
\end{lemma}

\begin{proof}
Restrict first to $S$; when $m=2$ this is the original germ. By the
general-position assumption,
\[
 x=F|_S,\qquad y=G|_S
\]
form a local coordinate system at the origin of $S$. Hence
\[
 \omega_0|_S=\frac{dy}{y}-\lambda\frac{dx}{x}.
\]
The coefficient calculation is the same as in the proof of the
two-dimensional relative-decomposition proposition, except that the
Laurent indices now have an arbitrary fixed lower bound. The only divisors
are $p-\lambda q$, with $(p,q)\in\mathbb Z^2$. The integer-index
diophantine estimate gives convergence after division and does not change
the finite lower bound. The only resonant coefficient is the constant one;
its two components give constants $\alpha,\beta$. We obtain directly on
$S$ the decomposition
\[
 \eta|_S=a_0\omega_0|_S+dh_0+\alpha\frac{dF|_S}{F|_S}
 +\beta\frac{dG|_S}{G|_S},
\]
with $a_0,h_0$ of finite pole order. If $m=2$, this is already the asserted
decomposition. If $m\geq3$, apply
\hyperref[thm:B]{Theorem~B}, which produces meromorphic germs $a,h$ of
finite pole order with the asserted decomposition.
\end{proof}

\begin{lemma}[Several-component finite-pole decomposition]
\label{lem:several-finite-pole}
Let
\[
 \omega_0=\sum_{j=1}^n\lambda_j\frac{df_j}{f_j}
\]
satisfy the hypotheses of \hyperref[thm:G]{Theorem~G}. Let $\eta$ be a
meromorphic $1$-form whose polar divisor is supported on
$D=(f_1\cdots f_n=0)$ and has arbitrary finite order. If
\[
 d\eta\wedge\omega_0=0,
\]
then there are meromorphic germs $a,h$, of finite pole order supported on
$D$, and constants $\mu_1,\ldots,\mu_n\in\C$ such that
\[
 \eta=a\omega_0+dh+\sum_{j=1}^n\mu_j\frac{df_j}{f_j}.
\]
\end{lemma}

\begin{proof}
Let
\[
 \pi:(M,E)\longrightarrow(\C^2,0)
\]
be an embedded resolution of $D$, chosen simultaneously with the reduction
of $\mathcal F_{\omega_0}$. Put
\[
 \widetilde\omega_0=\pi^*\omega_0,\qquad
 \widetilde\eta=\pi^*\eta,
\]
and denote by $\widetilde D$ the total transform. Since $\eta$ has finite
pole order along $D$, both $\widetilde\eta$ and its coefficients have finite
pole order along the normal-crossings divisor $\widetilde D$.

Let $C$ be a component of $\widetilde D$. Its residue for
$\widetilde\omega_0$ is
\[
 \rho_C=\sum_{j=1}^n m_{Cj}\lambda_j,\qquad
 m_{Cj}=\operatorname{ord}_C(f_j\circ\pi).
\]
The vector $(m_{C1},\ldots,m_{Cn})$ is nonzero, and therefore
$\rho_C\neq0$. At a corner $C\cap C'$, after absorbing a holomorphic exact
term into the coordinates, we have
\[
 \widetilde\omega_0=
 \rho_C\frac{dx}{x}+\rho_{C'}\frac{dy}{y}.
\]
Write the coefficients of $\widetilde\eta$ in Laurent series. They vanish
when either Laurent exponent is smaller than some fixed integer $-N_V$.
For a monomial $x^py^q$, the nonzero divisors in the relative coefficient
equations are
\[
 p\rho_{C'}-q\rho_C
 =\sum_{j=1}^n(pm_{C'j}-qm_{Cj})\lambda_j.
\]
The integer-index Diophantine condition bounds their reciprocals by a
polynomial in $|p|+|q|$. Consequently, coefficientwise division preserves
convergence on a smaller bidisc. It does not change the Laurent exponents,
so the resulting local functions have finite principal parts. At a smooth
point of $\widetilde D$, the meromorphic flow-box lemma gives the same
finite-pole conclusion. If the displayed divisor vanishes, then
$x^py^q$ is a meromorphic first integral of $\widetilde\omega_0$. The
corresponding coefficient of a relative primitive is therefore not
determined by the relative equation; we normalize it to be zero. Changing
this normalization adds a meromorphic first integral to the primitive and
does not change its relative differential. The remaining cokernel is the
constant logarithmic residue class treated below. We therefore obtain
local decompositions, up to constant logarithmic terms, on every chart of
a finite cover of $E$.

At a generic point of the strict transform of $(f_j=0)$, denote its
logarithmic coefficient by $\mu_j$. The local decomposition determines the
vector of logarithmic coefficients only modulo a common multiple of the
residue vector of $\widetilde\omega_0$; changing that multiple amounts
exactly to changing $a_V$. Fix this ambiguity on one chart and propagate
the normalization along the exceptional dual graph. Since this graph is a
tree, the choices are compatible on all overlaps. A meridian around an
exceptional component $C$ projects to the integral combination of the
original meridians with coefficients $m_{Cj}$. Hence the logarithmic
coefficient induced along $C$ is
\[
 \mu_C=\sum_{j=1}^n m_{Cj}\mu_j.
\]
It follows that
\[
 \Theta:=\widetilde\eta-
 \sum_{j=1}^n\mu_j\frac{d(f_j\circ\pi)}{f_j\circ\pi}
\]
has zero logarithmic obstruction modulo $\widetilde\omega_0$ along every
component of $\widetilde D$. More precisely,
\[
 \operatorname{Res}_C(\Theta)\in\mathbb C\rho_C
\]
for each component $C$, and at a corner $C\cap C'$ the residue vector is
proportional to $(\rho_C,\rho_{C'})$. These residual terms are absorbed
into the local multiple $a_V\widetilde\omega_0$. Thus, on each member $V$
of the chosen cover, there are meromorphic functions $a_V,h_V$ of finite
pole order such that
\[
 \Theta|_V=a_V\widetilde\omega_0+dh_V.
\]

It remains to glue the local primitives. Choose a transverse disk
$\Sigma$ in a regular flow box and let $h_\Sigma$ be one of the local
primitives. For a holonomy germ $g$ represented by leafwise paths
$\gamma_{g,z}$, define
\[
 c_g(z)=h_\Sigma(g(z))-h_\Sigma(z)
 -\int_{\gamma_{g,z}}\Theta.
\]
Then
\[
 c_{g\circ f}=c_f+c_g\circ f.
\]
Because $\widetilde\omega_0$ is closed logarithmic, the holonomy of each
invariant component is simultaneously linear in a logarithmic transverse
coordinate. If
$\mu=\exp(2\pi i\rho_{C'}/\rho_C)$ lies on the unit circle, choose for
each $k\neq0$ the nearest integer $r$ to $k\rho_{C'}/\rho_C$. Then
\[
 |\mu^k-1|\asymp
 \frac{|k\rho_{C'}-r\rho_C|}{|\rho_C|},
\]
where $r=O(|k|)$ and
$k\rho_{C'}-r\rho_C$ is an integral combination of the original residues
with coefficient norm $O(|k|)$. The integer-index estimate therefore gives
the required polynomial lower bound; the finite-multiplier case is handled
by averaging. The constant Laurent
coefficient of $c_g$ is the limiting logarithmic period of $\Theta$, hence
vanishes by the preceding residue calculation.

Finally, choose the cover of the compact exceptional divisor to be finite
and set $N=\max_V N_V$. Linear holonomy substitutions preserve Laurent
exponents, and leafwise integration does not introduce exponents below
$-N$. Thus all $c_g$ have a common finite principal part. The additive
holonomy lemma gives a meromorphic $b$ such that
\[
 c_g=b\circ g-b.
\]
Replacing $h_\Sigma$ by $h_\Sigma-b$ makes the continuation invariant under
all holonomy returns. Since the open leaves are simply connected, leafwise
continuation is independent of the chosen chain of flow boxes. The
non-nodal saturation lemma shows that this continuation reaches a punctured
neighborhood of the whole exceptional divisor. The local decompositions
above then extend it meromorphically across smooth points and corners, with
finite pole order. We obtain meromorphic functions
$\widetilde a,\widetilde h$ near $E$ satisfying
\[
 \Theta=\widetilde a\widetilde\omega_0+d\widetilde h.
\]
Meromorphic functions near the exceptional set descend through the proper
modification. The descended germs $a,h$ have finite pole order supported
on $D$ and give the asserted decomposition.
\end{proof}

\begin{proof}[Proof of Theorem G]
Put $P=f_1\cdots f_n$ and $\omega_t=\Omega_t/P$. Thus
\[
 \omega_t=\omega_0+\sum_{k\geq1}t^k\omega_k.
\]
The coefficient of $t$ in $\omega_t\wedge d\omega_t=0$ gives
$d\omega_1\wedge\omega_0=0$. By
Lemma~\ref{lem:several-finite-pole},
\[
 \omega_1=a_1\omega_0+dh_1+
 \sum_{j=1}^n c_j^1\frac{df_j}{f_j}.
\]
Divide by the formal unit $1+t a_1$. The coefficient of order one is then
closed, and the remaining coefficients are meromorphic with finite pole
order supported on $D$. At order two, integrability gives
\[
 d\widetilde\omega_2\wedge\omega_0=0,
\]
because the exact and logarithmic terms already extracted at order one are
closed. Applying Lemma~\ref{lem:several-finite-pole} to
$\widetilde\omega_2$ and dividing by $1+t^2a_2$ removes its multiple of
$\omega_0$.

Inductively, after the first $r-1$ steps, the normalized deformation has
the form
\[
 \omega_0+\sum_{k=1}^{r-1}t^k
 \left(dh_k+\sum_{j=1}^n c_j^k\frac{df_j}{f_j}\right)
 +t^r\widetilde\omega_r+O(t^{r+1}).
\]
All displayed coefficients before order $r$ are closed. The coefficient of
$t^r$ in the integrability equation therefore gives
$d\widetilde\omega_r\wedge\omega_0=0$. The several-component finite-pole
lemma gives
\[
 \widetilde\omega_r=a_r\omega_0+dh_r+
 \sum_{j=1}^n c_j^r\frac{df_j}{f_j},
\]
and division by $1+t^ra_r$ completes the induction step. Hence
\[
 \frac{\omega_t}{\widehat U}
 =\omega_0+d\widehat H+
 \sum_{j=1}^n\widehat c_j(t)\frac{df_j}{f_j}.
\]
Since $\omega_t=\Omega_t/P$, this is precisely
\[
 \frac{\Omega_t}{P\widehat U}
 =\omega_0+d\widehat H+
 \sum_{j=1}^n\widehat c_j(t)\frac{df_j}{f_j},
\]
the normalized formula stated in Theorem~G. Each
application of Lemma~\ref{lem:several-finite-pole} gives finite pole order
for the coefficient under consideration, but the order may increase with
$r$.
\end{proof}

\begin{proof}[Proof of Theorem H]
Put $\omega_t:=\Omega_t/(FG)$ and
\[
 \omega_t=\omega_0+\sum_{j=1}^{+\infty}t^j\omega_j.
\]
From $\omega_t\wedge d\omega_t=0$, we obtain
\[
 d\omega_1\wedge\omega_0=0.
\]
Therefore
\[
 \omega_1=a_1\omega_0+dh_1
 +\lambda_1\frac{dF}{F}+\mu_1\frac{dG}{G},
\]
and hence
\[
 \omega_t=(1+ta_1)\omega_0+tdh_1
 +t\left(\lambda_1\frac{dF}{F}+\mu_1\frac{dG}{G}\right)
 +\sum_{j=2}^{+\infty}t^j\omega_j.
\]
Set
\begin{align*}
 \widetilde\omega_t:=\frac{\omega_t}{1+ta_1}
 &=\omega_0+\frac{t}{1+ta_1}dh_1
 +\frac{t}{1+ta_1}\left(\lambda_1\frac{dF}{F}
 +\mu_1\frac{dG}{G}\right)
 +\sum_{j=2}^{+\infty}\frac{t^j}{1+ta_1}\omega_j\\
 &=\omega_0+tdh_1
 +t\left(\lambda_1\frac{dF}{F}+\mu_1\frac{dG}{G}\right)
 +\sum_{j=2}^{+\infty}t^j\widetilde\omega_j,
\end{align*}
where $\widetilde\omega_j$ is meromorphic, with finite pole order along
$(FG=0)$. Apply the finite-pole relative-decomposition lemma. Since
$\widetilde\omega_t\wedge d\widetilde\omega_t=0$, we have
\[
 \left(\omega_0+t\left(dh_1+\lambda_1\frac{dF}{F}
 +\mu_1\frac{dG}{G}\right)+t^2\widetilde\omega_2
 +t^3\widetilde\omega_3+\cdots\right)
 \wedge\left(t^2d\widetilde\omega_2
 +t^3d\widetilde\omega_3+\cdots\right)=0.
\]
Thus $d\widetilde\omega_2\wedge\omega_0=0$, and therefore
\[
 \widetilde\omega_2=a_2\omega_0+dh_2
 +\lambda_2\frac{dF}{F}+\mu_2\frac{dG}{G}.
\]
It follows that
\begin{align*}
 \widetilde\omega_t
 &=\omega_0+t\left[dh_1+\lambda_1\frac{dF}{F}
 +\mu_1\frac{dG}{G}\right]
 +t^2\left[a_2\omega_0+dh_2+\lambda_2\frac{dF}{F}
 +\mu_2\frac{dG}{G}\right]+t^3\widetilde\omega_3+\cdots\\
 &=(1+t^2a_2)\omega_0+tdh_1+t^2dh_2
 +(t\lambda_1+t^2\lambda_2)\frac{dF}{F}
 +(t\mu_1+t^2\mu_2)\frac{dG}{G}
 +t^3\widetilde\omega_3+\cdots.
\end{align*}
Dividing by $1+t^2a_2$, we obtain a new expression of the form
\[
 \omega_0+tdh_1+t^2dh_2
 +(t\lambda_1+t^2\lambda_2)\frac{dF}{F}
 +(t\mu_1+t^2\mu_2)\frac{dG}{G}
 +\sum_{j=3}^{+\infty}t^j\widetilde{\widetilde\omega}_j,
\]
where $\widetilde{\widetilde\omega}_j$ is meromorphic with finite pole
order along $(FG=0)$. Proceeding in this way, we obtain
\[
 \frac{\omega_t}{\displaystyle\prod_{j=1}^{+\infty}(1+a_jt^j)}
 =\omega_0+\sum_{j=1}^{\infty}t^jdh_j
 +\left(\sum_{j=1}^{\infty}\lambda_jt^j\right)\frac{dF}{F}
 +\left(\sum_{j=1}^{\infty}\mu_jt^j\right)\frac{dG}{G}.
\]
Put
\[
 \widehat U(t):=\prod_{j=1}^{+\infty}(1+a_jt^j),\qquad
 \widehat H_t:=\sum_{j=1}^{+\infty}t^jh_j,
\]
\[
 \widehat\lambda(t):=\sum_{j=1}^{+\infty}\lambda_jt^j,
 \qquad
 \widehat\mu(t):=\sum_{j=1}^{+\infty}\mu_jt^j.
\]
Then
\[
 \frac{\omega_t}{\widehat U(t)}
 =\omega_0+d\widehat H_t
 +\widehat\lambda(t)\frac{dF}{F}
 +\widehat\mu(t)\frac{dG}{G}.
\]
Since
\[
 \omega_0=\frac{dG}{G}-\lambda\frac{dF}{F},
\]
we obtain
\[
 \frac{\omega_t}{\widehat U(t)}
 =\bigl(\widehat\lambda(t)-\lambda\bigr)\frac{dF}{F}
 +\bigl(\widehat\mu(t)+1\bigr)\frac{dG}{G}
 +d\widehat H_t.
\]
Set
\[
 \widehat\alpha(t)=\widehat\lambda(t)-\lambda,
 \qquad
 \widehat\beta(t)=\widehat\mu(t)+1.
\]
Then $\widehat\alpha(0)=-\lambda$ and $\widehat\beta(0)=1$. Multiplying
the preceding identity by $FG\widehat U(t)$ gives
\[
 \Omega_t=\widehat U(t)\left[
 \widehat\alpha(t)G\,dF+\widehat\beta(t)F\,dG
 +FG\,d\widehat H_t
 \right],
\]
which is the asserted normal form.
\end{proof}

\begin{corollary}[Higher-dimensional several-component version]
\label{cor:higher-several}
Let $m\geq3$ and
\[
 \omega_0=\sum_{j=1}^n\lambda_j\frac{df_j}{f_j},
 \qquad P=f_1\cdots f_n,\qquad\Omega_0=P\omega_0,
\]
where $D=(P=0)$ is reduced. Assume that the residues are
$\mathbb Z$-linearly independent and satisfy the integer-index
Diophantine condition. Suppose that there is
a holomorphic embedding
\[
 \zeta:(\C^2,0)\longrightarrow(\C^m,0),\qquad S=\zeta(\C^2),
\]
in general position with respect to $(\mathcal F_{\omega_0},D)$ in the
sense of Definition~\ref{def:section-general-position}, such that:
\begin{enumerate}[label=\textup{(\roman*)}]
 \item the restricted divisor
 \[
  D_S=\left(\prod_{j=1}^n(f_j\circ\zeta)=0\right)
 \]
 is reduced;
 \item
 \[
  \omega_0|_S=
  \sum_{j=1}^n\lambda_j
  \frac{d(f_j\circ\zeta)}{f_j\circ\zeta}
 \]
 defines a non-dicritical logarithmic foliation germ;
 \item the open leaves of the restricted foliation are simply connected;
 \item the reduction of the restricted foliation is free of nodal
 separators.
\end{enumerate}

Then, for every analytic integrable deformation
\[
 \Omega_t=\Omega_0+\sum_{k=1}^{\infty}t^k\Omega_k,
\]
there are $\widehat U,\widehat H\in\mathcal M_m[[t]]$, with
$\widehat U|_{t=0}=1$, and series
$\widehat c_j(t)\in t\C[[t]]$ such that
\[
 \frac{\Omega_t}{P\widehat U}
 =\omega_0+d\widehat H+
 \sum_{j=1}^n\widehat c_j(t)\frac{df_j}{f_j}.
\]
Every coefficient of $\widehat U$ and $\widehat H$ has finite pole order
supported on $D$.
\end{corollary}

\begin{proof}
Restrict a finite-pole relative-cohomology equation to $S$ and apply
Lemma~\ref{lem:several-finite-pole}. Then
\hyperref[thm:B]{Theorem~B} extends the resulting decomposition to
$(\C^m,0)$. This gives the higher-dimensional
several-component finite-pole lemma. The formal induction in the proof of
Theorem~G then applies verbatim.
\end{proof}

As a final result we have: 

\begin{proposition}\label{prop:eight}
Let $F,G,\lambda$ and $\Omega_t$ be as in
\hyperref[thm:H]{Theorem~H}. There is a
codimension-two foliation germ $\mathcal F$ at
$0\in\C^m\times\C$ such that:
\begin{enumerate}[label=\textup{(\roman*)}]
 \item $\C^m\times\{t\}$ is invariant for every $t\in(\C,0)$;
 \item the restriction $\mathcal F|_{\C^m\times\{t\}}$ coincides with
 the foliation $\mathcal F_t:\Omega_t=0$, for every $t\in(\C,0)$.
\end{enumerate}
\end{proposition}

\begin{proof}
On $(\C^m\times\C,0)$ consider the Pfaff system
\[
 \mathcal I=\langle dt,\Omega\rangle,
 \qquad \Omega(x,t):=\Omega_t(x).
\]
Let $\mathcal I^{\mathrm{sat}}$ denote its saturation; this removes common
divisorial factors and is the Pfaff system defining $\mathcal F$ in the
reduced sense. At every point where $\mathcal I$ has the expected rank,
$\mathcal I^{\mathrm{sat}}$ and $\mathcal I$ define the same distribution.
At every regular point of $\Omega_t$, the forms $dt$ and $\Omega$ are
linearly independent, because $\Omega$ has no $dt$ component. Thus
$\mathcal I$ has codimension two in the total space. On the fiber
$\C^m\times\{t\}$ the equation $dt=0$ is automatic, and the restricted
system is the single equation $\Omega_t=0$; hence its restriction has
codimension one and coincides with $\mathcal F_t$.
After saturation, the restriction is the reduced Pfaff equation obtained
from $\Omega_t$ by removing its common divisorial factor, and therefore it
defines exactly the same foliation $\mathcal F_t$.

It remains to check integrability. Write $d_x$ for the differential in the
$x$ variables. Since every $\Omega_t$ is integrable,
\(
 \Omega_t\wedge d_x\Omega_t=0.
\)
Since
\[
 d\Omega=d_x\Omega+dt\wedge\frac{\partial\Omega}{\partial t},
\]
we have identically
\[
 dt\wedge\Omega\wedge d\Omega
 =dt\wedge\Omega\wedge d_x\Omega=0.
\]
Together with $d(dt)=0$, this is precisely the Frobenius condition for the
decomposable Pfaff system generated by $dt$ and $\Omega$. Integrability is
preserved by saturation, so $\mathcal I^{\mathrm{sat}}$ defines the asserted
codimension-two foliation germ. Since $dt$ belongs to its defining ideal,
every fiber $\C^m\times\{t\}$ is invariant.
\end{proof}

\end{document}